\documentclass[11pt]{amsart}
\usepackage{amssymb}
\usepackage{amsmath}
\usepackage{mathtools}
\usepackage{mdwtab}
\usepackage{subcaption}
\usepackage{stmaryrd}
\usepackage{tikz}
\usetikzlibrary{positioning,fit}
\usepackage[indent]{parskip}
\usepackage{url}
\usepackage{enumitem}
\usepackage{fullpage}
\usepackage{pstricks,pst-node}

\usepackage{bm}
\usepackage{threeparttable}
\usepackage{pifont}
\usepackage{amssymb}
\usepackage{caption}

\usepackage{graphicx}
\usepackage{genyoungtabtikz}
\usepackage{booktabs}
\usepackage{comment, amsgen, amscd, xspace, epsfig, float, epic}

\usepackage{geometry}
\usepackage{multirow}
\usepackage{tikz}
\usepackage{threeparttable}
\usepackage{enumitem}
\usepackage{fancyhdr}

\usepackage[style=alphabetic,sorting=nyt]{biblatex}
\usepackage[colorlinks,linkcolor=blue,urlcolor=cyan,citecolor=blue]{hyperref}
\theoremstyle{definition}
\newtheorem{Cor}[equation]{Corollary}
\newtheorem{Exam}[equation]{Example}

\newtheorem{Def}[equation]{Definition}
\newtheorem{Thm}[equation]{Theorem}

\newtheorem{Lem}[equation]{Lemma}
\newtheorem{Rem}[equation]{Remark}
\newtheorem{Prop}[equation]{Proposition}

\newcommand{\hobox}[3]{\draw (0+#1,0-#2) rectangle (1+#1,-1-#2)++(-0.5,+0.5) node {$ #3$};}

\numberwithin{equation}{section}

\DeclareMathAlphabet\mathsf{OT1}{cmss}{m}{n}

\newcommand{\ev}{\mathrm{ev}}
\newcommand{\odd}{\mathrm{od}}

\newcommand{\Ann}{\mathrm{Ann}}
\newcommand{\caO}{\mathcal{O}}

\renewcommand{\labelenumi}{(\theenumi)}

\begin{document}

\title[Annihilator varieties]{On the annihilator variety of a highest weight Harish-Chandra module}

\author{Zhanqiang Bai}
\address{Department of Mathematical Sciences, Soochow University, Suzhou 215006, China}
\email{\tt  zqbai@suda.edu.cn}

\author{Jing Jiang*}
\address{
School of mathematical Sciences, Shanghai Key Laboratory of Pure Mathematics and Mathematical Practice, East China Normal University, Shanghai 200241, China} 
\email{jjsd6514@163.com}

\begin{abstract}
Let $G$ be a Hermitian type Lie group with maximal compact subgroup $K$. Let  $L(\lambda)$ be a highest weight Harish-Chandra module of $G$ with the infinitesimal character  $\lambda$. By using some combinatorial algorithm, we obtain a description of the annihilator variety of $L(\lambda)$. As an application, when $L(\lambda)$ is unitarizable, we prove that 
the Gelfand-Kirillov dimension of $L(\lambda)$ only depends on the value of $z=(\lambda,\beta^{\vee})$, where $\beta$ is the highest root.
\end{abstract}
\footnote{*Corresponding author:  jjsd6514@163.com}


\maketitle
% \setcounter{tocdepth}{1}
% \tableofcontents

\section{Introduction}
% 	\section{Introduction}
% \begin{spacing}{1.1}
Let $G$ be a  simple Lie group with complexified Lie algebra $\mathfrak{g}$ and $K$ be a maximal 
compact subgroup   with complexified Lie algebra $\mathfrak{k}$ such that $G / K$ is a Hermitian symmetric space. We have a Cartan decomposition $\mathfrak{g}=\mathfrak{k}\oplus \mathfrak{p}$ and a decomposition $\mathfrak{p}=\mathfrak{p}^+\oplus\mathfrak{p}^-$ into nonzero irreducible $K$-subrepresentations. Then we have a triangular decomposition $\mathfrak{g}=\mathfrak{k}\oplus\mathfrak{p}^+\oplus\mathfrak{p}^-$ and $\mathfrak{q}=\mathfrak{k}\oplus\mathfrak{p}^+$ is a maximal  parabolic subalgebra of $\mathfrak{g}$.
Let $\mathfrak{h}$ be  a Cartan subalgebra of $\mathfrak{g}$. Denote by $\Phi$ the root system associated with $(\mathfrak{g},\mathfrak{h})$, fix a positive system $\Phi^+$, let $\mathbb{C}_{\lambda-\rho}$ be a finite-dimensional irreducible $\mathfrak{k}$-module with highest weight $\lambda -\rho\in\mathfrak{h}^*$, where $\rho$ is half the sum of positive roots.
% unitary representation of $G$ is called a unitary highest weight representation if its Harish-Chandra $(\mathfrak{g} , K )$-module is a simple quotient of a Verma module of $\mathfrak{g}$. The Harish-Chandra $(\mathfrak{g}, K)$-module of a unitary highest weight representation of $G$ is called a unitary highest weight module. Then we have the usual decomposition $\mathfrak{g}=\mathfrak{k}\oplus\mathfrak{p}^+\oplus\mathfrak{p}^-$ as a $\mathfrak{k}$-module. Then $\mathfrak{q}=\mathfrak{k}\oplus\mathfrak{p}^+$ is a parabolic subalgebra of $\mathfrak{g}$. Denote by $\Phi$ the root system associated with $(\mathfrak{g},\mathfrak{h})$, fix a positive system $\Phi^+$, let $F(\lambda)$ be a finite-dimensional irreducible $\mathfrak{k}$-module with highest weight $\lambda -\rho\in\mathfrak{h}^*$, where $\rho$ is half the sum of positive roots. 
It can also be viewed as a
	$\mathfrak{q}$-module with trivial $\mathfrak{p}^+$-action. The corresponding generalized Verma module $N(\lambda)$ is defined by
	\[
	N(\lambda):=U(\mathfrak{g})\otimes_{U(\mathfrak{q})}\mathbb{C}_{\lambda-\rho}.
	\]
	The simple quotient of $N(\lambda)$ is denoted by $L(\lambda)$, which is a highest weight module with highest weight $\lambda-\rho$. From \cite{EHW}, we call $L(\lambda)$ a highest weight Harish-Chandra module. The classification of unitarizable highest weight Harish-Chandra modules is given in \cite{EHW} and \cite{Ja1} independently.

For a $U(\mathfrak{g})$-module $M$, let $I(M)=\Ann(M)$ be its annihilator ideal in  $U(\mathfrak{g})$ and $J(M)$ be the corresponding graded ideal in $ S(\mathfrak{g})=\text{gr} (U(\mathfrak{g}))$. The zero set of $J(M)$ in the dual vector space $\mathfrak{g}^*$ of $\mathfrak{g}$ is called {\it the annihilator variety} of $M$, which is also called {\it the associated variety} of $I(M)$. We usually denote it by $V(\Ann(M))$. The study of associated varieties of  primitive ideals (or annihilator varieties of highest weight modules) is a very important problem since it is closely related with many research fields, such as representations of Weyl groups, Kazhdan-Lusztig cells and  representations of Lie groups, see for example \cites{BB,BV82,BMSZ-ABCD,GSK}. 
	 Borho-Brylinski \cite{BB} proved that the associated variety of a primitive ideal $I$ with a fixed regular integral infinitesimal
	character is the Zariski closure of the nilpotent orbit in $\mathfrak{g}^*$ attached to $I$, via the Springer correspondence. Joseph \cite{AJ6} extended this result to a primitive ideal with a general infinitesimal  character.

Recently Bai-Ma-Wang \cite{BMW} found some simple algorithms to characterize the nilpotent orbit appearing in the annihilator variety of a highest weight module for all classical type Lie algebras by using  the Robinson-Schensted algorithm. In this paper, our first goal is to give the explicit description of  $V(\Ann(L(\lambda)))$ for all  highest weight Harish-Chandra modules of classical types by using the partition algorithm in \cite{BMW}.
See Theorem \ref{AV-a}, Theorem \ref{Bpartition}, Theorem \ref{C-AV}, Theorem \ref{so(2,2n-2)} and Theorem \ref{D-AV}. 

The Gelfand-Kirillov (GK) dimension is an important invariant to measure the size of an infinite-dimensional $U(\mathfrak{g})$-module $M$, see \cite{AJ1} and \cite{Vo78}.
From Joseph \cite{AJ1}, we know that $\dim V(\Ann(L(\lambda)))=2\mathrm{GKdim} L(\lambda)$. 
When there is only one nilpotent orbit with the  dimension $2\mathrm{GKdim} L(\lambda)$, we can determine this nilpotent orbit appearing in $V(\Ann(L(\lambda)))$ by the Gelfand-Kirillov dimension of $L(\lambda)$.
In \cite{BXX}, Bai-Xiao-Xie computed the Gelfand-Kirillov dimensions of highest weight Harish-Chandra modules of  exceptional  types.
Then we can determine $V(\Ann(L(\lambda)))$ by the Gelfand-Kirillov dimension of $L(\lambda)$ since there is only  one nilpotent orbit with the  dimension $2\mathrm{GKdim} L(\lambda)$ for these cases.
See Theorem \ref{e6-GK} and Theorem \ref{e7-GK}.
Except for the cases of $SU(p,q)$, $Sp(n,\mathbb{R})$ and $SO^*(2n)$, our characterization of $V(\Ann(L(\lambda)))$ is given in a function  of $z=(\lambda,\beta^{\vee})$.

 Our second goal of this paper is to prove that when $L(\lambda)$ is a unitary highest weight module with highest weight $\lambda-\rho=\lambda_0+z\xi$, where $z=(\lambda,\beta^{\vee})$, $(\lambda_0+\rho,\beta)=0$, $\xi$ is orthogonal to the root system $\Phi(\mathfrak{k})$ and $(\xi,\beta^{\vee})=1$, then ${\rm GKdim}\:L(\lambda)$ is independent of the selection of $\lambda_0$. In other words,  ${\rm GKdim}\:L(\lambda)$ only depends on the value of $z=(\lambda,\beta^{\vee})$. This result was firstly proved in \cite{Bai-Hu} by using previous results in \cite{Jo92}. In our paper, we use our characterization of $V(\Ann(L(\lambda)))$ to give a new proof for this result. See Theorem \ref{kconst}.

% The diagram is as follows.
% \begin{center}	
% \begin{tikzpicture}
% 	%\caption{reducible points of unitary highest weight module}
% 	\draw (-2,0)--(3,0);
% 	\filldraw[gray] (3,0) circle (2pt) (3,-0.25) node[below] {$ a(\lambda_0)$};
% 	\filldraw[gray] (4,0) circle (2pt);
% 	\filldraw[gray] (5,0) circle (2pt);
% 	\filldraw[gray] (5.5,0) circle (1pt);
% 	\filldraw[gray] (5.75,0) circle (1pt);
% 	\filldraw[gray] (6,0) circle (1pt);
% 	\filldraw[gray] (6.25,0) circle (1pt);
% 	\filldraw[gray] (7,0) circle (2pt) (7,-0.25) node[right] {$ b(\lambda_0)$};
% 	\draw (-2,-1)--(4,-1);
% 	\filldraw[gray] (4,-1) circle (2pt) (4,-1.25) node[below] {$ a(\mu_0)$};
% 	\filldraw[gray] (5,-1) circle (2pt);
% 	\filldraw[gray] (5.75,-1) circle (1pt);
% 	\filldraw[gray] (6,-1) circle (1pt);
% 	\filldraw[gray] (6.25,-1) circle (1pt);
% 	\filldraw[gray] (6.5,-1) circle (1pt);
% 	\filldraw[gray] (7,-1) circle (2pt);
% 	\filldraw[gray] (8,-1) circle (2pt) (8,-1.25) node[below] {$ b(\mu_0)$};
% 	\draw[dashed,rounded corners] (3.7,0.5) rectangle (7.3,-1.45);	
% 	\draw [<->] (4,-0.1)--(4,-0.8);
% 	\draw [<->] (5,-0.1)--(5,-0.8);
% 	\draw [<->] (7,-0.1)--(7,-0.8);
% \end{tikzpicture}

% \textbf{Figure 1}
	
% \end{center}
% If $\lambda_0$ turns into $\mu_0$, the corresponding unitary reducible point will change. However, the ${\rm GK}$-dimension of the vertical corresponding unitary reducible points in the box is unchanged, which has been proved by Bai-Markus in \cite{BM}.

The paper is organized as follows. In \S \ref{pre}, we give some necessary preliminaries about Gelfand-Kirillov dimension, associated variety, annihilator variety and partition algorithm for highest weigh modules. In \S \ref{ann-sln}-\ref{ann-2e}, we give the characterization of the  annihilator  variety $V(\Ann(L(\lambda)))$ for  a highest weight Harish-Chandra module of  type $A_{n-1}$, type $B_n$ and $C_n$,  type $D_n$, 
  type $E_6$ and $E_7$ respectively.
In \S \ref{GKdim-k}, we prove that when $L(\lambda)$ is a unitary highest weight module, then ${\rm GKdim}\:L(\lambda)$ only depends on the value of $z=(\lambda,\beta^{\vee})$.
% \end{spacing}

\section{Preliminaries}\label{pre}
In this section, we will give some brief preliminaries on GK dimension, associated variety, annihilator variety and partition algorithm for highest weight modules. See \cites{Vo78,VJ,BMW} for more details.

\subsection{Gelfand-Kirillov dimension and associated variety}

Let $\mathfrak{g}$ be a simple complex Lie algebra and $\mathfrak{h}$ be a Cartan subalgebra.
Let $M$ be a finitely generated $U(\mathfrak{g})$-module. Fix a finite dimensional generating subspace $M_0$ of $M$. Let $U_{n}(\mathfrak{g})$ be the standard filtration of $U(\mathfrak{g})$. Set $M_n=U_n(\mathfrak{g})\cdot M_0$ and
\(
{\rm gr} (M)=\bigoplus\limits_{n=0}^{\infty} {\rm gr}_n M,
\)
where ${\rm gr}_n M=M_n/{M_{n-1}}$. Thus ${\rm gr}(M)$ is a graded module of ${\rm gr}(U(\mathfrak{g}))\simeq S(\mathfrak{g})$.

\begin{Def} The {Gelfand-Kirillov dimension} of $M$  is defined by
	\begin{equation*}
		\operatorname{GKdim} M = \varlimsup\limits_{n\rightarrow \infty}\frac{\log\dim( U_n(\mathfrak{g})M_{0} )}{\log n}.
	\end{equation*}
\end{Def}

%The definition of GK dimensio$n$ is independent of  the choice of $M_0$.

\begin{Def}[\cite{VJ}]
	The  {associated variety} of $M$ is defined by
	\begin{equation*}
		V(M):=\{X\in \mathfrak{g}^* \mid f(X)=0 {\rm~for~ all~} f\in \operatorname{Ann}_{S(\mathfrak{g})}(\operatorname{gr} M)\}.
	\end{equation*}
\end{Def}

The above two definitions are independent of the choice of $M_0$, and $\dim V(M)=	{\rm GKdim}\: M$ (e.g., \cite{NOT}). 

\begin{Def} Let $\mathfrak{g}$ be a finite-dimensional simple Lie algebra. Let $I$ be a two-sided ideal in $U(\mathfrak{g})$. Then ${\rm gr}(U(\mathfrak{g})/I)\simeq S(\mathfrak{g})/\text{gr}I$ is a graded $S(\mathfrak{g})$-module. Its annihilator is ${\rm gr}I$. We define its associated variety by
	$$V(I):=V(U(\mathfrak{g})/I)=\{X\in \mathfrak{g}^* \mid p(X)=0\ \mbox{for all $p\in {\rm gr}I$}\}.
	$$
\end{Def}

\subsection{Associated varieties of highest weight Harish-Chandra modules}

In \cite{Bai-Hu}, Bai-Hunziker have found the formula of  the  Gelfand-Kirillov dimensions of unitary highest weight modules. We adopt their notations here.

Let $(-, -): \mathfrak{h} \times \mathfrak{h}^* \to \mathbb{C}$ be the canonical pairing.
Let $\beta$ denote the highest positive root. Let $\rho$ be the half sum of positive roots. The dual Coxeter number $h^\lor$ of the root system $\Phi$ is defined by
$$h^\lor:=(\rho,\beta^\lor)+1.$$
where $\beta^\lor$ is the coroot of $\beta$. Equivalently, $h^\lor$ can be defined as the sum of coefficients of highest short roots in $\Phi^+$ when it is written as a sum of simple roots. Let $r$ be the $\mathbb{R}$-rank of $G$, i.e., the dimension of a Cartan subgroup of the (real) group $G$. And we have the following table.

% \begin{align*}
% 	&\small{\textbf{ Table~2 }  \text { Some constants associated to } G/K}\\
% 	&\begin{array}{lcccc}	
% 		\toprule[1pt] 
% 		\operatorname{Lie}(G) & r & c~& &~h^{\vee}-1 \\
% 		\midrule[0.6pt]		
% 		 \mathfrak{su}(p, n-p) & \min \{p, n-p\} & 1~&~ &~n-1 \\
% 		\mathfrak{s p}(n, \mathbb{R}) & n & 1 / 2~&~ &~~n \\
% 		\mathfrak{s o}^*(2n), n=2 m & m & 2~& ~&~~2 n-3 \\
% 		\mathfrak{s o}^{*}(2n), n=2 m+1 & m & 2~& ~&~~2 n-3 \\
% 		\mathfrak{s o}(2,2n-1) & 2 & n-3 / 2~& ~&~~2 n-2 \\
% 		\mathfrak{s o}(2,2n-2) & 2 & n-2~& ~&~~2 n-3 \\
% 		\mathfrak{e}_6(-14) & 2 & 3 &~ &~11 \\
% 		\mathfrak{e}_7(-25) & 3 & 4 & ~&~17 \\	
% 		\bottomrule[1pt]
% 	\end{array}
% \end{align*}

% \begin{table}[h!]
% \centering

% \renewcommand{\arraystretch}{1.4}
% \setlength\tabcolsep{5pt}
% \begin{tabular}{lllll}
% %\begin{tabular}{|l|l|l|l|}
% \toprule
%   $\operatorname{Lie}(G)$ &   $r$ & $c$ & $h^\vee-1$ \\
% \midrule
% $\mathfrak{su}(p,q)$, $n=p+q-1$ & $\min\{p,q\}$ &  $1$ & $n$  \\ % \hline
%     $\mathfrak{sp}(n,\mathbb{R})$  & $n$ &   $1/2$ & $n$   \\ % \hline
%     $\mathfrak{so}^{*}(2n)$, $n=2m$  & $m$  & $2$ & $2n-3$ \\ %\hline
%    $\mathfrak{so}^{*}(2n)$, $n=2m+1$  & $m$  & $2$ & $2n-3$ \\ %\hline
%    $\mathfrak{so}(2,2n-1)$  & $2$  &  $n-3/2$ & $2n-2$ \\ %\hline
%   $\mathfrak{so}(2,2n-2)$  & $2$ &  $n-2$& $2n-3$ \\ %\hline
%   $\mathfrak{e}_{6(-14)}$   & $2$ &  $3$ & $11$ \\  %\hline
%   $\mathfrak{e}_{7(-25)}$  & $3$ &  $4$ & $17$ \\
% \bottomrule
% \end{tabular}
% \caption{Some constants associated to $G/K$}
% \label{constants}
% \end{table}

\begin{table}[htbp]
	\centering
	\renewcommand{\arraystretch}{1.4}
	\setlength\tabcolsep{5pt}
	\caption{Some constants associated to $G/K$}
\label{constants}
{	\begin{tabular}{llll}		
  \toprule
  $G$ &   $r$ & $c$  & $h^\vee-1$ \\ \hline 
%\midrule
 $SU(p,q)$, $n=p+q-1$ & $\min\{p,q\}$ &  $1$ & $n$  \\  
    $Sp(n,\mathbb{R})$  & $n$ &   $1/2$ & $n$   \\  
    $SO^{*}(2n)$, $n=2m$  & $m$  & $2$ & $2n-3$ \\ 
   $SO^{*}(2n)$, $n=2m+1$  & $m$  & $2$ & $2n-3$ \\ 
   $SO(2,2n-1)$  & $2$  &  $n-3/2$ & $2n-2$ \\ 
  $SO(2,2n-2)$  & $2$ &  $n-2$ & $2n-3$ \\ 
  $E_{6(-14)}$   & $2$ &  $3$ & $11$ \\  
  $E_{7(-25)}$  & $3$ &  $4$ & $17$ \\
\bottomrule

\end{tabular}}

\end{table}

They also have proved that if $L(\lambda)$ is a unitary highest weight module with highest weight $\lambda-\rho$, then there is an integer $0 \leqslant k(\lambda) \leqslant r$ such that the associated variety of $L(\lambda)$ is $\overline{\mathcal{O}}_{k(\lambda)}$. We list the two crucial theorems.
\begin{Prop}[{\cite[Thm. 1.1 and Cor. 5.2]{Bai-Hu}}]\label{k}Suppose $L(\lambda)$ is  a unitary highest weight module  with highest weight $\lambda-\rho$.
Let $0 \leqslant k(\lambda) \leqslant r$ be the integer such that $V(L(\lambda))=\overline{\mathcal{O}}_{k(\lambda)}$. Then $k(\lambda)=r$ if $-(\lambda-\rho,\beta^{\vee})>c(r-1)$ and $k(\lambda)=\frac{-(\lambda-\rho, \beta^{\vee})}{c}$
if $-(\lambda-\rho,\beta^{\vee})\leq c(r-1)$.
Also we have
$$
\operatorname{dim} \overline{\mathcal{O}}_{k(\lambda)}=k(\lambda)\left(h^{\vee}-1\right)-k(\lambda)(k(\lambda)-1)c.
$$
In particular, $\operatorname{dim} \mathfrak{p}^{+}=r\left(h^{\vee}-1\right)-r(r-1) c$.
\end{Prop}

\begin{Prop}[{\cite[Thm. 6.2]{Bai-Hu}}]\label{unitary}
Suppose $L(\lambda)$ is a unitary highest weight module  with highest weight $\lambda-\rho$. Define  $z(\lambda):=(\lambda,\beta^{\vee})$ and $z_k=(\rho, \beta^{\vee})-kc$. Then
	\[	{\rm GKdim}\:L(\lambda)=\left\{
\begin{array}{ll}
	rz_{r-1}, & \textnormal{if}\:z(\lambda)<z_{r-1}, \\	     	  	   
	kz_{k-1}, & \textnormal{if}\:z(\lambda)=z_k {\rm~with~} 1 \leqslant k \leqslant r-1,\\
	0, & \textnormal{if}\:z(\lambda)=z_0.
\end{array}	
\right.
\]	

\end{Prop}

In \cite{BXX}, Bai-Xiao-Xie have given the description of the associated variety of $L(\lambda)$, we recall some notations and results.

For  a totally ordered set $ \Gamma $, we  denote by $ \mathrm{Seq}_n (\Gamma)$ the set of sequences $ x=(x_1,x_2,\cdots, x_n) $   of length $ n $ with $ x_i\in\Gamma $. In our paper, we usually take $\Gamma$ to be $\mathbb{Z}$ or a coset of $\mathbb{Z}$ in $\mathbb{C}$. By using the Robinson-Schensted algorithm for $x\in  \mathrm{Seq}_n (\Gamma)$, we can get a Young tableau $P(x)$. Denote ${\bf p}(x)=\mathrm{sh}(P(x))=[p_1,p_2,...,p_N]$, where $p_i$ is the number of boxes in the $i$-th
row of $P(x)$.

\begin{Def}Fix a positive integer $n$, a partition of $n$ is a decreasing sequence ${\bf p}=[p_1,p_2,...,p_n]$ of nonnegative integers such that $\sum_{1\leq i\leq n}p_i=n$.
	We say ${\bf q}=[q_1,q_2,\cdots,q_N]$ is the dual partition of a partition ${\bf p}=[p_1,p_2,\cdots,p_N]$ and write ${\bf q}={\bf p}^t$ if $q_i$ is the length of $i$-th column of the Young tableau $P$ with shape ${\bf p}$.
\end{Def}

For a Young diagram $P$, use $ (k,l) $ to denote the box in the $ k $-th row and the $ l $-th column.
We say the box $ (k,l) $ is {even} (resp. {odd}) if $ k+l $ is even (resp. odd). Let $ p_i ^{\ev}$ (resp. $ p_i^{\odd} $) be the number of even (resp. odd) boxes in the $ i $-th row of the Young diagram $ P $.
One can easily check that \begin{equation}\label{eq:ev-od}
	p_i^{\ev}=\begin{cases}
		\left\lceil \frac{p_i}{2} \right\rceil,&\text{ if } i \text{ is odd},\\
		\left\lfloor \frac{p_i}{2} \right\rfloor,&\text{ if } i \text{ is even},
	\end{cases}
	\quad p_i^{\odd}=\begin{cases}
		\left\lfloor \frac{p_i}{2} \right\rfloor,&\text{ if } i \text{ is odd},\\
		\left\lceil \frac{p_i}{2} \right\rceil,&\text{ if } i \text{ is even}.
	\end{cases}
\end{equation}
Here for $ a\in \mathbb{R} $, $ \lfloor a \rfloor $ is the largest integer $ n $ such that $ n\leq a $, and $ \lceil a \rceil$ is the smallest integer $n$ such that $ n\geq a $. For convenience, we set
\begin{equation*}
	{\bf p}^{\ev}=(p_1^{\ev},p_2^{\ev},\cdots)\quad\mbox{and}\quad {\bf p}^{\odd}=(p_1^{\odd},p_2^{\odd},\cdots).
\end{equation*}

% One can easily check that
% \begin{equation}\label{eq:ev-od}
% 	q_2=\begin{cases}
% 		2q_2^{\odd}-1&\text{ if } q_2 \text{ is odd},\\
% 		2q_2^{\odd}&\text{ if } q_2 \text{ is even},
% 	\end{cases}
% 	\quad q_2=\begin{cases}
% 		2q_2^{\ev}+1&\text{ if } q_2 \text{ is odd},\\
% 		2q_2^{\ev}&\text{ if } q_2 \text{ is even}.
% 	\end{cases}
% \end{equation}

\begin{Exam}
	Let ${\bf p}=[6,5,4,3,2,1]$ be the shape of the Young diagram $P$. Then odd and even boxes in $P$ are marked as follows.
	\[
	{\begin{tikzpicture}[scale=0.6,baseline=-40pt]
			\hobox{0}{0}{E}
			\hobox{0}{1}{O}
			\hobox{0}{2}{E}
			\hobox{0}{3}{O} 
			\hobox{0}{4}{E}	
			\hobox{0}{5}{O}		
			\hobox{1}{0}{O}
			\hobox{1}{1}{E}
			\hobox{1}{2}{O}
			\hobox{1}{3}{E} 
			\hobox{1}{4}{O}
			\hobox{2}{0}{E}
			\hobox{2}{1}{O}
			\hobox{2}{2}{E}
			\hobox{2}{3}{O} 
			\hobox{3}{0}{O}
			\hobox{3}{1}{E}
			\hobox{3}{2}{O}
			\hobox{4}{0}{E}
			\hobox{4}{1}{O}
			\hobox{5}{0}{O}		
	\end{tikzpicture}}
	\] 
	Then ${\bf p}^{\ev}=(3,2,2,1,1,0)$ and ${\bf p}^{\odd}=(3,3,2,2,1,1)$.
	
\end{Exam}

For convenience, $ x=(x_1,x_2,\cdots,x_n)\in \mathrm{Seq}_n (\Gamma) $, set
\begin{equation*}
	\begin{aligned}
		{x}^-=&(x_1,x_2,\cdots,x_{n-1}, x_n,-x_n,-x_{n-1},\cdots,-x_2,-x_1).
		%{}^-{x}=&(-x_n,-x_{n-1},\cdots, -x_2,-x_1,x_1,x_2,\cdots, x_{n-1}, x_n).
	\end{aligned}
\end{equation*}

\begin{Prop}[{\cite[Thm. 6.1, Thm. 6.2]{BXX}}]\label{BCD}
		Let $L(\lambda)$ be a highest weight Harish-Chandra module of $G$ with highest weight $\lambda-\rho$ and $\lambda=(t_1,\cdots,t_n)\in\mathfrak{h}^*$. Set ${\bf q}={\bf q}(\lambda)=(q_1,q_2,\cdots,q_{2n})={\bf p}(\lambda)^t$ when $G$ is of type $A$ 
 and ${\bf q}={\bf q}(\lambda^-)=(q_1,q_2,\cdots,q_{2n})={\bf p}(\lambda^-)^t$ when $G$ is of type $B, C$ or $D$. Then $V(L(\lambda))=\overline{\mathcal{O}}_{k(\lambda)}$ with $k(\lambda)$ given as follows.
\begin{enumerate}		
	\renewcommand{\labelenumi}{(\theenumi)}	
\item $G=SU(p,n-p)$. Then
\[	k(\lambda)=\left\{
	\begin{array}{ll}
		q_2, & \textnormal{~if~$\lambda$~is~integral}, \\	     	  	   
	\min\{p,q\}, & \textnormal{~otherwise}.
	\end{array}	
	\right.
	\]		
 
\item $G=Sp(n, \mathbb{R})$ with $n \geq 2$. Then
$$
k(\lambda)= \begin{cases}2 q_{2}^{\odd}, & \textnormal{ if } t_{1} \in \mathbb{Z}, \\ 2 q_{2}^{\ev}+1, & \textnormal{ if } t_{1} \in \frac{1}{2}+\mathbb{Z}, \\ n, & \textnormal{ otherwise. }\end{cases}
$$

\item $G=S O^{*}(2 n)$ with $n \geq 4$. Then
$$
k(\lambda)= \begin{cases}q_{2}^{\ev}, & \textnormal{ if } t_{1} \in \frac{1}{2} \mathbb{Z}, \\ \left\lfloor\frac{n}{2}\right\rfloor, & \textnormal{ otherwise. }\end{cases}
$$
\item $G=S O(2,2 n-1)$ with $n \geq 3$. Then
$$
k(\lambda)= \begin{cases}0, & \textnormal{ if } t_{1}-t_{2} \in \mathbb{Z}, t_{1}>t_{2}, \\ 1, & \textnormal{ if } t_{1}-t_{2} \in \frac{1}{2}+\mathbb{Z}, t_{1}>0, \\ 2, & \textnormal{ otherwise. }\end{cases}
$$
\item $G=S O(2,2 n-2)$ with $n \geq 4$. Then
$$
k(\lambda)= \begin{cases}0, & \textnormal{ if } t_{1}-t_{2} \in \mathbb{Z}, t_{1}>t_{2}, \\ 1, & \textnormal{ if } t_{1}-t_{2} \in \mathbb{Z},-\left|t_{n}\right|<t_{1} \leq t_{2}, \\ 2, & \textnormal{ otherwise. }\end{cases}
$$		
\end{enumerate}			
		
\end{Prop}

\begin{Prop}[{\cite[Thm. 7.1]{BXX}}]\label{E}
Let $L(\lambda)$ be a highest weight Harish-Chandra module of  $E_{6(-14)}$ or $E_{7(-25)}$. If $\lambda\in\mathfrak{h}^{\ast}$ is not integral, then $V(L(\lambda))=\overline{\caO}_{k(\lambda)}$ with $k(\lambda)=2$ (for $E_{6(-14)}$) or 3 (for $E_{7(-25)}$). In the case when $\lambda$ is integral, $k(\lambda)$ is given as follows.
\begin{enumerate}	
     \item If $G=E_{6(-14)}$, then
     	\[	k(\lambda)=\left\{
     	\begin{array}{ll}
     	\vspace{1ex}
     		0, & \textnormal{~if~$\Psi_{\lambda}^+\cap S_2\ne\emptyset$}, \\	\vspace{1ex}     	  	   
     	1, & \textnormal{~if~$\Psi_{\lambda}^+\cap S_1\ne\emptyset$~and~$\Psi_{\lambda}^+\cap S_2=\emptyset$},\\  
     	2, & \textnormal{~if~$\Psi_{\lambda}^+\cap S_1=\emptyset$}.
     	\end{array}	
     	\right.
     	\]	
Here $S_1=\big\{\frac{1}{2}(\pm(e_1+e_2+e_3-e_4)-e_5-e_6-e_7+e_8)\}=\big\{\beta_1,\beta_2\}$, $S_2=\big\{\frac{1}{2}(e_1-e_2-e_3-e_4-e_5-e_6-e_7+e_8)\}=\big\{\alpha_1\}$ and $\Psi_{\lambda}^+=\big\{\alpha\in\Phi^+|(\lambda,\alpha^{\vee})>0\}$.     
     
     \item If $G=E_{7(-25)}$, then
     	\[	k(\lambda)=\left\{
          	\begin{array}{ll}
          	\vspace{1ex}
          		0, & \textnormal{~if~$\Psi_{\lambda}^+\cap S_3\ne\emptyset$}, \\	\vspace{1ex}     	  	   
          	1, & \textnormal{~if~$\Psi_{\lambda}^+\cap S_2\ne\emptyset$~and~$\Psi_{\lambda}^+\cap S_3=\emptyset$},\\  \vspace{1ex}  
          	2, & \textnormal{~if~$\Psi_{\lambda}^+\cap S_1\ne\emptyset$~and~$\Psi_{\lambda}^+\cap S_2=\emptyset$},\\ \vspace{1ex}  
          	3, & \textnormal{~if~$\Psi_{\lambda}^+\cap S_1=\emptyset$}.
          	\end{array}	
          	\right.
          	\]	
Here $S_1=\big\{\frac{1}{2}(\pm(e_1-e_2-e_3+e_4)-e_5+e_6-e_7+e_8),e_5+e_6\}=\big\{\beta_3,\beta_4,\beta_5\}$, $S_2=\big\{\pm e_1+e_6\}=\big\{\beta_6,\beta_7\}$, $S_3=\big\{-e_5+e_6\}=\big\{\beta_8\}$ and $\Psi_{\lambda}^+=\big\{\alpha\in\Phi^+|(\lambda,\alpha^{\vee})>0\}$.

\end{enumerate}

\end{Prop}

\subsection{Annihilator variety} 
 For
a $U(\mathfrak{g})$-module $M$, let $I(M) = \Ann(M)$ be its annihilator ideal in $U(\mathfrak{g})$ and $J(M)$
be the corresponding graded ideal in $S(\mathfrak{g}) = \text{gr}(U(\mathfrak{g}))$. The zero set of $J(M)$ in the
dual vector space $\mathfrak{g}^{\ast}$ is called the {\bf annihilator variety} of $M$, and we denote it by $V (\Ann(M))$. Let $G$ be a connected semisimple finite dimensional complex algebraic group with Lie algebra $\mathfrak{g}$ and $W$ be the Weyl group of $\mathfrak{g}$.  We use $L_w$ to denote the simple highest weight $\mathfrak{g}$-module of highest weight $-w\rho-\rho$ with  $w\in W$.  We denote $I_w=\Ann(L_w)$, then by \cite{BoB3} $V(I_w)=\overline{\mathcal{O}}_w$ is irreducible, where ${\mathcal{O}}_w$ is  a  nilpotent coadjoint orbit in $\mathfrak{g}^{\ast}$. 

\begin{Prop}[\cite{AJ6}]
	Let $\mathfrak{g}$ be a reductive Lie algebra and $I$ be a primitive ideal in $U(\mathfrak{g})$.Then $V(I)$ is the closure of a single nilpotent coadjoint orbit $\mathcal{O}_I$ in $\mathfrak{g}^*$. In particular, for a highest weight module $L(\lambda)$, we have $V(\Ann( L(\lambda)))=\overline{\mathcal{O}}_{\Ann(L(\lambda))}$.
\end{Prop}

From \cite[Prop. 2.7]{AJ1} and \cite{Vo78}, we have $\dim V(I_w)=2\mathrm{GKdim}(L_w)$. In general, we have the following corollary.
\begin{Cor}\label{dim}
	Suppose $L(\lambda)$  is a highest weight Harish-Chandra module with $V(L(\lambda))=\overline{\mathcal{O}}_{k(\lambda)}$ for some $0\leq k(\lambda)\leq r$. Then 
$\dim{\mathcal{O}}_{\Ann(L(\lambda))}=2\dim{\mathcal{O}}_{k(\lambda)}=2\dim V(L(\lambda))$.
\end{Cor}

In \cite{BMW} the authors give the annihilator varieties of classical Lie algebras, and we recall the notions and algorithms here.

For $\lambda \in \mathfrak{h}^*$, write $\lambda=\left(t_1, \ldots, t_n\right)=\sum\limits_{i=1}^n t_i e_i$, where $t_i \in \mathbb{C}$ and $\left\{e_i \mid 1 \leq i \leq n\right\}$ is the canonical basis of the Euclidean space $\mathbb{R}^n$. We associate to $\lambda$ a set $S(\lambda)$ of some Young tableaux as follows. Let $\lambda_Y: t_{i_1}, t_{i_2}, \ldots, t_{i_r}$ be a maximal subsequence of $t_1, t_2, \ldots, t_n$ such that $t_{i_k}, 1 \leq k \leq r$ are congruent to each other by $\mathbb{Z}$. Then the Young tableau $P\left(\lambda_Y\right)$ associated to the subsequence $\lambda_Y$ by using RS algorithm is a Young tableau in $S(\lambda)$.

% \begin{Lem}[\cite{BMW}, Theorem 5.4]
% Let $\mathfrak{g}=\mathfrak{s l}(n, \mathbb{C})$. Suppose $\lambda \in \mathfrak{h}^*$. Then
% $$
% V(\operatorname{Ann}(L(\lambda)))=\overline{\mathcal{O}}_{p(\lambda)},
% $$
% where $p(\lambda)$ is the partition of the Young tableau
% $$
% P(\lambda)=\underset{P\left(\lambda_Y\right) \in S(\lambda)}{\stackrel{c}{\sqcup}} P\left(\lambda_Y\right) .
% $$

% \end{Lem}
In case of $\mathfrak{g}=\mathfrak{so}(2n+1,\mathbb{C})$, $\mathfrak{sp}(n,\mathbb{C})$ and $\mathfrak{so}(2n,\mathbb{C})$,  we need some more notations before we give the descriptions of $V(\Ann(L(\lambda)))$.  

Define $[\lambda] $ to be the set of  maximal subsequence $ \lambda_Y$ of  $ \lambda $ such that any two entries of $ \lambda_Y$ have an integral  difference or sum. In this case, we set $ [\lambda]_1 $ (resp. $ [\lambda]_2 $)  to be the subset of $ [\lambda] $ consisting of sequences with  all entries belonging to $ \mathbb{Z} $ (resp. $ \frac{1}{2}+\mathbb{Z} $), 	Since there is at most one element in $[\lambda]_1 $ and $[\lambda]_2 $, we denote them by  $\lambda_{(0)}$ and $\lambda_{(\frac{1}{2})}$. Set $[\lambda]_{1,2}=[\lambda]_1\cup [\lambda]_2$ and $[\lambda]_3=[\lambda]\setminus[\lambda]_{1,2}$. For $ \lambda_Y\in[\lambda]_{1,2}$  we can get a Young tableau $P(\lambda^-_Y)$ and  $P(\tilde{\lambda}_Y)$  for $  \lambda_Y\in[\lambda]_3$. 	We use $P{\stackrel{c}{\sqcup}}Q$ to denote a new Young tableau whose columns are the union of columns of the Young tableaux $P$ and $Q$. We use $X$ to denote the corresponding type of Lie algebras, i.e., $X=B,C$ or $D$.	
Then we have the following result.

\begin{Prop}[\cite{BMW}, Thm. 5.4,  6.5, 6.6 and  7.14]\label{Annv}
Suppose $\lambda\in\mathfrak{h}^{\ast}$, the annihilator variety of $L(\lambda)$ is given  as follows.
 \begin{enumerate}
\item If $\mathfrak{g}=\mathfrak{s l}(n, \mathbb{C})$, then
$$
V(\operatorname{Ann}(L(\lambda)))=\overline{\mathcal{O}}_{{\bf p}(\lambda)},
$$
where ${\bf p}(\lambda)$ is the partition of the Young tableau
$$
P(\lambda)=\underset{P\left(\lambda_Y\right) \in S(\lambda)}{\stackrel{c}{\sqcup}} P\left(\lambda_Y\right).
$$

\item If  $\mathfrak{g} = \mathfrak{so}(2n+1, \mathbb{C}), \mathfrak{sp}(n, \mathbb{C})$ or $\mathfrak{so}(2n, \mathbb{C})$, $\lambda\in \mathfrak{h}^*$ and
		$[\lambda]=[\lambda]_{1} \cup [\lambda]_{2}\cup [\lambda]_3$
		with $[\lambda]_3=\{{\lambda}_{{Y}_1},\dots,{\lambda}_{{Y}_m}\}$. Let
		\begin{enumerate}
			\item $\mathbf{p}_{0}$ be the $X$-type special partition associated to
			$[\lambda]_{1}$;
			\item ${\mathbf p}_{\frac{1}{2}}$ be the $C$-type special  partition (for type $B$) or $C$-type metaplectic special partition (for types $C$ and $D$) associated to
			$[\lambda]_{2}$;
			\item $\mathbf{p}_{i}$ be the $A$-type partition associated to
			$\tilde{\lambda}_{Y_i}$.
		\end{enumerate}
		Let ${\bf p}_{_X}(\lambda)$ be the $X$-collapse of
		\begin{equation*}
		\mathbf{d}_{\lambda} := \mathbf{p}_{0} {\stackrel{c}{\sqcup}} {\mathbf p}_{\frac{1}{2}} {\stackrel{c}{\sqcup}}({{\stackrel{c}{\sqcup}}}_i 2 \mathbf{p}_i).
		\end{equation*}
		Then we have
		\[
		V(\Ann (L(\lambda)))=\overline{\mathcal{O}}_{{\bf p}_{_X}(\lambda)}.
		\]

 \end{enumerate}
\end{Prop}

\begin{Rem}
   From \cite[\S 6]{BMW}, we  know that $\mathbf{p}_{0}$ will be the partition $[1]$ in the case of type $B$ if $[\lambda]_1=\emptyset$.

    Note that a $C$-type metaplectic
special partition ${\bf p}$ means that ${\bf p}=(({\bf d}^t)_D)^t$, where ${\bf d}$ is a special partition of type $D$. From \cite{BMW},  we can get a $C$-type metaplectic
special partition from the Young tableau $P(\lambda_{(\frac{1}{2})}^-)$ by using the H-algorithm defined in \cite[Chap. 8]{BMW}.  We can also get a special partition from the Young tableau $P(\lambda^-)$ or $P(\lambda_{(\frac{1}{2})}^-)$ by using the H-algorithm. 

Roughly speaking, the H-algorithm is equivalent to the following: during the process of collapse and expansion, we can not move the odd boxes for types B and
C (resp. even boxes  for type D) and the moving even box can not meet another even box in
the same row (resp. the moving odd box can not meet another odd box in the same row for type D).
\end{Rem}
Special partitions are given by the following characterizations.

\begin{Prop}[{\cite[Thm. 5.1.1, 5.1.2, 5.1.3 $\&$ 5.1.4 and Prop. 6.3.7]{CM}}]\label{bcd}
The nilpotent orbits of classical types can be identified with some partitions as follows:
\begin{enumerate}
    \item Nilpotent orbits in $\mathfrak{sl}(n,\mathbb{C})$ are in one-to-one correspondence with the set of  partitions of $n$. Every partition is special.
    \item Nilpotent orbits in $\mathfrak{so}{(2n+1,\mathbb{C})}$ are in one-to-one correspondence with the set of partitions of $2n+1$ in which even parts occur with even multiplicity. A partition $\bf q$ of type $B$ is special if and only if its dual partition ${\bf q}^t$ is a partition of type $B$.
    \item Nilpotent orbits in $\mathfrak{sp}{(n,\mathbb{C})}$ are in one-to-one correspondence with the set of partitions of $2n$ in which odd parts occur with even multiplicity. A partition $\bf q$ of type $C$ is special if and only if its dual partition ${\bf q}^t$ is a partition of type $C$.
    \item Nilpotent orbits in $\mathfrak{so}{(2n,\mathbb{C})}$ are in one-to-one correspondence with the set of partitions of $2n$ in which even parts occur with even multiplicity,
  except that each ``very even" partition ${\bf d}$ (consisting of only even parts) correspond to two orbits, denoted by $\mathcal{O}^I_{\bf d}$ and $\mathcal{O}^{II}_{\bf d}$. A partition $\bf q$ of type $D$ is special if and only if its dual partition ${\bf q}^t$ is a partition of type $C$.
\end{enumerate}
\end{Prop}

\section{Annihilator varieties   for type \texorpdfstring{$A_{n-1}$}{}}\label{ann-sln}
% In \cite{DW}, we know that the nilpotent orbits in $\mathfrak{sl_n}$ are in one-to-one correspondence with the partitions of $n$, which is naturally special, and we make use of partitions to describe the associated variety of $\Ann(L(\lambda))$.
In this section, we consider the case of type $A_{n-1}$. 
\begin{Def}
For $\mathfrak{g}=\mathfrak{sl}(n,\mathbb{C})$, we say $\lambda\in \mathfrak{h}^{\ast} $ is $(p, q)$-dominant if $t_i-t_j\in\mathbb{Z}_{>0}$ for all $i, j$ such that $1\le i < j\le p $ and $ p + 1\le i < j\le p + q$, where $\lambda= (t_1,t_2, \dots,t_n)$. In particular, $t_1 >t_2 >\dots >t_p$ and $ t_{p+1} >t_{p+2} > \dots >t_{p+q}$.
\end{Def}

\begin{Prop}[{\cite[Thm. 5.2]{BX1}}]\label{dominant}
For $\mathfrak{g}=\mathfrak{sl}(n,\mathbb{C})$, assume that $\lambda\in \mathfrak{h}^{\ast} $ is $(p, q)$-dominant.
	\begin{enumerate}
		\renewcommand{\labelenumi}{(\theenumi)}	
		\item If $t_p-t_{p+1}\in\mathbb{Z}$,
		that is, $\lambda$ is an integral weight, then $P(\lambda)$ is a Young tableau
		with at most two columns. And in this case ${\rm GKdim}\:L(\lambda)= m(n-m)$ where  $m$ is the number of entries in the second column of $P(\lambda$).	
		
		\item  If $t_p-t_{p+1}\notin\mathbb{Z}$,
		then $S(\lambda)$ consists of two Young tableaux with single column, and in this case ${\rm GKdim}\:L(\lambda)= pq$.
		
	\end{enumerate}
\end{Prop}

\begin{Thm}\label{AV-a}
		Let $L(\lambda)$ be a Harish-Chandra  module of $SU(p,n-p)$ with highest weight $\lambda-\rho\in\mathfrak{h}^*$. Suppose $q_2$ is  the length of second column of the Young
tableau $P(\lambda)$. Then $V(\Ann(L(\lambda)))=\overline{\mathcal{O}}_{{\bf p}}$ with
		\[	{\bf p}=\left\{
		\begin{array}{ll}
			[2^{q_2},1^{n-2q_2}], & \textnormal{~if~$\lambda$~is~integral},\\	   	  \lbrack 2^r,1^{n-2r}\rbrack,  & \textnormal{~otherwise}.
		\end{array}	
		\right.
		\]		
\end{Thm}

\begin{proof}

From \cite{EHW}, we know that  $\lambda\in \mathfrak{h}^{\ast} $ is $(p, n-p)$-dominant.

% $\lambda=\lambda_0+z\xi$ with $\lambda_0=(\lambda_1,\lambda_2,\cdots,\lambda_n)$, $\lambda_1\geq\lambda_2\geq\cdots\geq\lambda_p$, $\lambda_{p+1}\geq\cdots\geq\lambda_n$ and $\lambda_i-\lambda_j\in\mathbb{Z}$ for $1\le i<j\le p$ and $p+1\le i<j\le n$. By the normalization $\langle\lambda_0+\rho, \beta\rangle=0 $, we have $\lambda_1-\lambda_n+n-1=0$. 

% In \cite{EHW}, we can obtain that $\xi=(\underbrace{a,\dots,a}_{p},\underbrace{a-1,\dots,a-1}_{n-p})$ for $a=\dfrac{n-p}{n}$, and
% $$\rho=(\frac{n-1}{2},\frac{n-3}{2},\dots,\frac{-n+3}{2},\frac{-n+1}{2}).$$	
% $$\lambda=(\lambda_1+za,\cdots,\lambda_p+za,\lambda_{p+1}+(a-1)z,\cdots,\lambda_n+(a-1)z).$$
% Thus
% $$\scriptsize{\!\lambda+\rho=(\underbrace{\lambda_1+za+\frac{n-1}{2},\cdots,\lambda_p+za+\frac{n-2p+1}{2}}_p,\underbrace{\lambda_{p+1}+(a-1)z+\frac{n-2p-1}{2},\cdots,\lambda_n+(a-1)z+\frac{1-n}{2}}_{n-p})\!}.$$

%  Suppose that $\lambda_p-\lambda_{p+1}\in\mathbb{Z}$. 
 
 \begin{enumerate}[leftmargin=18pt]
 	\renewcommand{\labelenumi}{(\theenumi)}
  \item When $\lambda $ is not integral, 
 % equivalently $\lambda_p+za+\frac{n-2p+1}{2}-\lambda_{p+1}-(a-1)z-\frac{n-2p-1}{2}=\lambda_p-\lambda_{p+1}+z+1\notin \mathbb{Z}$, we will have $z\notin \mathbb{Z}$. In this case, 
 $S(\lambda)$ has two Young tableaux and these two Young tableaux have only one column, so
 $${\bf p}(\lambda_{Y_1})=[1^p]\text{~and~}{\bf p}(\lambda_{Y_2})=[1^{n-p}].$$
 
By Proposition \ref{Annv}, ${\bf p}={\bf p}(\lambda)={\bf p}(\lambda_{Y_1}){\stackrel{c}{\sqcup}}{\bf p}(\lambda_{Y_2})=[2^r,1^{n-2r}]$ is the corresponding partition for $V(\Ann(L(\lambda)))=\overline{\mathcal{O}}_{{\bf p}}$, where $r=\min\{p, n-p\}$. 

% Of course, we can also draw the conclusion by  Corollary \ref{dim} and Theorem \ref{direct}  by computing the $\text{GKdim}L(\lambda)$.
% \begin{align*}
%  	{\rm GKdim}\:L(\lambda)&=\frac{n(n-1)}{2}-\frac{p(p-1)}{2}-\frac{(n-p)(n-p-1)}{2}\\&=-\frac{p^2}{2}+np-\frac{p^2}{2}=np-p^2=p(n-p).
%  \end{align*}
% Consider the partition $d=[2^r,1^{n-2r}]$, we know that $\dim(\overline{O}_d)=2p(n-p)$ by Lemma \ref{partition}. In this case $d=[2^r,1^{n-2r}]$ is the corresponding partition by Corollary \ref{dim} and Theorem \ref{direct}.

\item When $\lambda $ is integral, $P(\lambda)$ is a Young tableau
		with at most two columns by Proposition \ref{dominant}. Thus by Proposition \ref{Annv}, we have $V(\Ann(L(\lambda)))=\overline{\mathcal{O}}_{{\bf p}}$, where ${\bf p}={\bf p}(\lambda)=[2^{q_2},1^{n-2q_2}]$ with $q_2$ being  the length of second column of the Young
tableau $P(\lambda)$.

% equivalently $\lambda_p+za+\frac{n-2p+1}{2}-\lambda_{p+1}-(a-1)z-\frac{n-2p-1}{2}=\lambda_p-\lambda_{p+1}+z+1\in \mathbb{Z}$, we will have $z\in \mathbb{Z}$ and
% $\lambda_0+z\xi+\rho=\lambda +\rho$ is $(p,q)$-dominant. By Lemma \ref{dominant}, we know that ${\rm GKdim}\:L(\lambda) = m(n-m)$, where $ m $ is the number of entries in the second column of $P(\lambda+\rho)$.	

% By Lemma \ref{k} and Lemma \ref{BCD}, we know that the associated variety of $L(\lambda)$ is $\overline{\mathcal{O}}_{k(\lambda)}$ and $\dim(\overline{\mathcal{O}}_{k(\lambda)})=q_2(n-q_2)$. Then $\dim(\overline{\mathcal{O}}_{\Ann( L(\lambda))})=\dim(\overline{\mathcal{O}}_{d^{'}})=2q_2(n-q_2)$, where $d^{'}$ is a partition. 

% Consider the partition $d=[2^{q_2},1^{n-q_2}]$, by Lemma \ref{dim} we can calculate that $\dim(\overline{\mathcal{O}}_{d})=2q_2(n-q_2)$, then the associated variety of $\Ann(L(\lambda))$ must be $\overline{\mathcal{O}}_{d}$ by Lemma \ref{Annv} and Lemma \ref{direct}.
\end{enumerate}	
So far, we have completed the proof.
\end{proof}
\begin{Cor}
    Keep notations as above. When $V(L(\lambda))=\overline{\mathcal{O}}_{k(\lambda)}$, we will have 
$V(\Ann(L(\lambda)))=\overline{\mathcal{O}}_{{\bf p}}$ with
		\[	{\bf p}=
			[2^{k(\lambda)},1^{n-2k(\lambda)}].
		\]		
\end{Cor}

% \begin{Exam}
%     Let $ \mathfrak{g}= \mathfrak{sl}(4,\mathbb{C})$, suppose $\lambda=(5,3,4,2)$. Then 
% \[P(\lambda)=
%     \small{\begin{tikzpicture}[scale=\domscale+0.1,baseline=-28pt]
% 		\hobox{0}{0}{2}
% 		\hobox{1}{0}{4}
% 		\hobox{0}{1}{3}
% 		\hobox{0}{2}{5}
% \end{tikzpicture}}
% \]
% Therefore $d=[2,1^2]$, so the annihilator variety is $V(\Ann(L(\lambda)))=\overline{\mathcal{O}}_{[2,1^2]}$.
    
% \end{Exam}

\section{Annihilator varieties  for  types \texorpdfstring{$B_{n}$}{} and \texorpdfstring{$C_{n}$}{}}\label{ann-bc}

In this section, we consider the case of   types $B_n$ and $C_n$. 

\begin{Thm}\label{Bpartition}
	Let $L(\lambda)$ be a Harish-Chandra module of $SO(2,2n-1)$ with highest weight $\lambda-\rho=\lambda_0+z\xi\in \mathfrak{h}^*$. Denote $\lambda_0=(\lambda_1,\lambda_2,\cdots,\lambda_n)$, then $V(\Ann(L(\lambda)))=\overline{\mathcal{O}}_{{\bf p}}$ with
	\[	{\bf p}=\left\{
	\begin{array}{ll}
		[1^{2n+1}], & \textnormal{~if~$z\in \lambda_2-\lambda_1+\mathbb{Z}_{\geq 0}$},\\	
		 \lbrack 2^2,1^{2n-3}\rbrack,  & \textnormal{~if~$z> -\lambda_1-n+\frac{1}{2}\textnormal{~and~}z\in\frac{1}{2}+\mathbb{Z}$},\\
		  	  \lbrack 3,1^{2n-2}\rbrack,  & \textnormal{~otherwise}.
	\end{array}	
	\right.
	\]		
\end{Thm}
\begin{proof}	
	 From \cite{EHW} we know that $\lambda-\rho=\lambda_0+z\xi$ with $\lambda_0=(\lambda_1,\lambda_2,\cdots,\lambda_n)$,  $\lambda_2\geq\lambda_3\geq\cdots\geq\lambda_n\geq 0$,  $\lambda_i-\lambda_j\in\mathbb{Z}\textnormal{~and~}2\lambda_i\in\mathbb{N}$ for $2\le i,\:j\le n$. By the normalization $(\lambda_0+\rho, \beta)=0 $, we have $\lambda_1+\lambda_2=-2n+2$.

 \begin{enumerate}[leftmargin=18pt]
 	\item When $\lambda $ is integral,  we have $\xi=(1,0,\dots,0)$ and 
	\begin{align*}
	   \lambda=(z+\lambda_1+n-\frac{1}{2},\lambda_2+n-\frac{3}{2},\cdots,\lambda_n+\frac{1}{2}):=(t_1,...,t_n).
	\end{align*}
Thus we have   $t_2> t_3>\cdots> t_{n-1}> t_n>0$.

 When $t_1>t_2$ (equivalently $z\in \lambda_2-\lambda_1+\mathbb{Z}_{\geq 0}$), $L(\lambda)$ will be finite-dimensional and $V(\Ann(L(\lambda)))=\overline{\mathcal{O}}_{{\bf p}}$ with the partition ${\bf p}=[1^{2n+1}]$ corresponding to the trivial orbit of type $B_n$. 

When $t_1\leq t_2$, by using the Robinson-Schensted   algorithm for $\lambda^-$,  we can get  a Young tableau $P(\lambda^-)$ which consists of
at most three columns.
When $t_2\geq t_1>-t_n$, from the construction process, we can see that $P(\lambda^-)$ will be a Young tableau consisting of two columns with $c_1(P(\lambda^-))=2n-2$ and $c_2(P(\lambda^-))=2$. Thus ${\bf p}(\lambda^-)=[2,2,1^{2n-4}]$. In this case, the special partition of type $B$ corresponding to ${\bf p}(\lambda^-)$ is ${\bf p}=[3, 1^{2n-2}]$. Therefore $V(\Ann(L(\lambda)))=\overline{\mathcal{O}}_{{\bf p}}$ by Proposition \ref{Annv}.
When $t_1\leq -t_n$, $P(\lambda^-)$
 will be a Young tableau consisting of three columns with $c_1(P(\lambda^-))=2n-2$, $c_2(P(\lambda^-))=1$ and $c_3(P(\lambda^-))=1$. Thus ${\bf p}(\lambda^-)=[3,1^{2n-3}]$. In this case, the special partition of type $B$ corresponding to ${\bf p}(\lambda^-)$ is ${\bf p}=[3, 1^{2n-2}]$.

  \item When $\lambda $ is half integral, we have $t_n> 0$, $t_1-t_2\in \frac{1}{2}+\mathbb{Z}$, $2t_k\in  \mathbb{Z}$ for $1\leq k\leq n $ and $t_i-t_j \in \mathbb{Z}_{> 0}$ for $2\leq i<j\leq n$. Then we have $t_2>t_3>...>t_n>-t_n>...>-t_2 $. We denote $\lambda_{Y_1}=(t_2,...,t_n)$ and $\lambda_{Y_2}=(t_1)$. Thus $\lambda_{Y_1}, \lambda_{Y_2}\in [\lambda]_{1,2}$.
By using Robinson-Schensted  algorithm, we can get a Young tableau $P(\lambda_{Y_1}^-)$ with shape ${\bf p}(\lambda_{Y_1}^-)=[1^{2n-2}]$ and a Young tableau $P(\lambda_{Y_2}^-)$ with shape 
\[	{\bf p}(\lambda_{Y_2}^-)=\left\{
		\begin{array}{ll}
			[2], & \textnormal{~if~}t_1\leq 0,\\	   	  \lbrack 1,1\rbrack,  & \textnormal{~if~}t_1> 0.
		\end{array}	
		\right.
		\]		

If $t_1\in \frac{1}{2}+\mathbb{Z}$ (equivalently $z\in\frac{1}{2}+\mathbb{Z}$), we will have $\lambda_{Y_1}\in [\lambda]_{1}$ and $ \lambda_{Y_2}\in [\lambda]_{2}$. From Proposition \ref{Annv}, we have $\mathbf{p}_{0}=[1^{2n-1}]$ and 
\[	\mathbf{p}_{\frac{1}{2}}=\left\{
		\begin{array}{ll}
			[1,1], & \textnormal{~if~}t_1> 0,\\	   	  \lbrack 2\rbrack,  & \textnormal{~if~}t_1\leq 0.
		\end{array}	
		\right.
		\]		
Therefore by Proposition \ref{Annv}, we have $V(\Ann(L(\lambda)))=\overline{\mathcal{O}}_{{\bf p}}$ with 
\[	{\bf p}=(\mathbf{p}_{0} {\stackrel{c}{\sqcup}} {\mathbf p}_{\frac{1}{2}})_B=\left\{
		\begin{array}{ll}
			[2,2,1^{2n-3}], & \textnormal{~if~}t_1> 0,\\	   	  \lbrack 3,1^{2n-2}\rbrack,  & \textnormal{~if~}t_1\leq 0.
		\end{array}	
		\right.
		\]		
% $d=(\mathbf{p}_{0} {\stackrel{c}{\sqcup}} {\mathbf p}_{\frac{1}{2}})_B=[2,2,1^{2n-3}]$ if $t_1>0$ or $d=(\mathbf{p}_{0} {\stackrel{c}{\sqcup}} {\mathbf p}_{\frac{1}{2}})_B=[3,1^{2n-2}]$ if $t_1\leq 0$  

If $t_1\in \mathbb{Z}$  (equivalently $z\in\frac{1}{2}+\mathbb{Z}$), we will have $\lambda_{Y_1}\in [\lambda]_{2}$ and $ \lambda_{Y_2}\in [\lambda]_{1}$. From Proposition \ref{Annv}, we have $\mathbf{p}_{\frac{1}{2}}=[1^{2n-2}]$ and 
\[	\mathbf{p}_{0}=\left\{
		\begin{array}{ll}
			[1,1,1], & \textnormal{~if~}t_1> 0,\\	   	  \lbrack 3\rbrack,  & \textnormal{~if~}t_1\leq 0.
		\end{array}	
		\right.
		\]		
Therefore by Proposition \ref{Annv}, we have $V(\Ann(L(\lambda)))=\overline{\mathcal{O}}_{{\bf p}}$ with 
 \[	{\bf p}=(\mathbf{p}_{0} {\stackrel{c}{\sqcup}} {\mathbf p}_{\frac{1}{2}})_B=\left\{
		\begin{array}{ll}
			[2,2,1^{2n-3}], & \textnormal{~if~}t_1> 0,\\	   	  \lbrack 3,1^{2n-2}\rbrack,  & \textnormal{~if~}t_1\leq 0.
		\end{array}	
		\right.
		\]		
 
Note that $t_1>0$ is equivalent to $z>-\lambda_1-n+\frac{1}{2}$.

\item When $\lambda$ is not integral or half integral, we have $\lambda_{Y_1}=(t_2,\cdots,t_n)\in [\lambda]_{1,2}$ and $\lambda_{Y_2}=(t_1)\in [\lambda]_3$. By using Robinson-Schensted  algorithm, we can get a Young tableau $P(\lambda_{Y_1}^-)$ with shape ${\bf p}(\lambda_{Y_1}^-)=[1^{2n-2}]$ and a Young tableau $P(\lambda_{Y_2})$ with shape $[1]$. Thus we have $\mathbf{p}_{0}=[1^{2n-1}]$ or ${\mathbf p}_{\frac{1}{2}}=[1^{2n-2}]$, and $\mathbf{p}_{1}=[1]$. Therefore by Proposition \ref{Annv}, we have $V(\Ann(L(\lambda)))=\overline{\mathcal{O}}_{{\bf p}}$ with ${\bf p}=(\mathbf{p}_{0} {\stackrel{c}{\sqcup}} {2\mathbf p}_{1})_B=[3,1^{2n-2}]$ or ${\bf p}=([1]{\stackrel{c}{\sqcup}}\mathbf{p}_{\frac{1}{2}} {\stackrel{c}{\sqcup}} {2\mathbf p}_{1})_B=[3,1^{2n-2}]$.

\end{enumerate}

This finishes the proof.

\end{proof}

\begin{Rem}
    From \cite[Lem. 3.17]{EHW}, we know that the generalized Verma module $N(\lambda)$ will be irreducible when $\lambda$ is not integral or half integral. Thus we will have $V(\Ann(L(\lambda)))=V(\Ann(N(\lambda)))=\overline{\mathcal{O}}_{{\bf p}}$ with ${\bf p}=[3,1^{2n-2}]$, which is just the Richardson nilpotent  orbit  $G\cdot
    \mathfrak{p}^+$, see \cite{BZ}.
\end{Rem}

\begin{Thm}\label{C-AV}
	Let $L(\lambda)$ be a Harish-Chandra module of $Sp(n,\mathbb{R})$ with  highest weight $\lambda-\rho\in\mathfrak{h}^*$. 
 Suppose $q_2$ is  the length of second column of the Young
tableau $P(\lambda^-)$.
 Denote $\lambda_0=(\lambda_1,\lambda_2,\cdots,\lambda_n)$, then $V(\Ann(L(\lambda)))=\overline{\mathcal{O}}_{{\bf p}}$ with

	\[	{\bf p}=\left\{
	\begin{array}{ll}
		[2^{2q_2^{\odd}},1^{2n-4q_2^{\odd}}], & \textnormal{~if~}z\in\mathbb{Z},\\
		\lbrack 2^{2q_2^{\ev}+1},1^{2n-4q_2^{\ev}-2}\rbrack, & \textnormal{~if~}z\in\frac{1}{2}+\mathbb{Z},\\
		\lbrack 2^{n}\rbrack, &
		\textnormal{~otherwise.} \\			
	\end{array}	
	\right.
	\]

\end{Thm}
\begin{proof}
	From \cite{EHW} we know that $\lambda-\rho=\lambda_0+z\xi$ with $\lambda_0=(\lambda_1,\lambda_2,\cdots,\lambda_n)$,   $\lambda_i-\lambda_j\in\mathbb{N}$ for $i<j$. By the normalization $(\lambda_0+\rho, \beta)=0 $, we have $\lambda_1=-n$. 
	 
	Firstly we suppose that $n$ is even. Then we have the follows.
	
% \begin{align}\label{4.9}
% 	\lambda=(z-n,z+\lambda_2,\cdots,z+\lambda_n):=(t_1,t_1,\cdots,t_n).
% \end{align}
\begin{enumerate}[leftmargin=18pt]
    \item When $\lambda$ is integral, we have $\xi=(\underbrace{1,1,\dots,1}_{n})$ and $$\lambda=(\lambda_1+z+n,\lambda_2+z+n-1,\cdots,\lambda_n+z+1):=(t_1,t_1,\cdots,t_n).$$
Thus we have   $t_1> t_2>\cdots> t_{n-1}> t_n$.

When $t_1\in\mathbb{Z}$ (equivalently $z\in\mathbb{Z}$), by using the Robinson-Schensted algorithm for $\lambda^-$,
  we can get  a Young tableau $P(\lambda^-)$ which consists of
at most two columns with $c_2(P(\lambda^-))=q_2$ and $c_1(P(\lambda^-))=2n-q_2$. Thus ${\bf p}(\lambda^-)=[2^{q_2},1^{2n-2q_2}]$. In this case the special partition of type $C$ corresponding to ${\bf p}(\lambda^-)$ is
\[	{\bf p}=\left\{
	\begin{array}{ll}
		[2^{q_2},1^{2n-2q_2}], & \textnormal{~if~}q_2=2k,\\
	\lbrack 2^{q_2+1},1^{2n-2q_2-2}\rbrack, & \textnormal{~if~} q_2=2k+1.		
	\end{array}	
	\right.
	\]	
Equivalently,
${\bf p}=[2^{2q_2^{\odd}},1^{2n-4q_2^{\odd}}]$ by  (\ref{eq:ev-od}). Therefore $V(\Ann(L(\lambda)))=\overline{\mathcal{O}}_{{\bf p}}$ by Proposition \ref{Annv}.
 
\item When $\lambda$ is half integral (equivalently $z\in\frac{1}{2}+\mathbb{Z}$), we have $\lambda\in [\lambda]_2$. Similarly we can get  a Young tableau $P(\lambda^-)$ which consists of
at most two columns with $c_2(P(\lambda^-))=q_2$ and $c_1(P(\lambda^-))=2n-q_2$. Thus ${\bf p}(\lambda^-)=[2^{q_2},1^{2n-2q_2}]$. 
In this case the $C$-type metaplectic
special partition ${\mathbf p}_{\frac{1}{2}}$ corresponding to ${\bf p}(\lambda^-)$ is 
$${\mathbf p}_{\frac{1}{2}}=[2^{2q_2^{\ev}+1},1^{2n-4q_2^{\ev}-2}]=\left\{
	\begin{array}{ll}
		[2^{q_2+1},1^{2n-2q_2-2}], & \textnormal{~if~}q_2=2k,\\
	\lbrack 2^{q_2},1^{2n-2q_2}\rbrack, & \textnormal{~if~} q_2=2k+1.		
	\end{array}	
	\right.$$ by (\ref{eq:ev-od}). 
 
 Therefore $V(\Ann(L(\lambda)))=\overline{\mathcal{O}}_{{\bf p}}$ with ${\bf p}=({\mathbf p}_{\frac{1}{2}})_C={\mathbf p}_{\frac{1}{2}}$ by Proposition \ref{Annv}.

\item When $\lambda$ is not integral or half integral, we have $\lambda\in [\lambda]_3$. By Proposition \ref{Annv},  we have ${\bf p}_1=p(\tilde{\lambda})=p(\lambda)=[1^n]$.  Therefore $V(\Ann(L(\lambda)))=\overline{\mathcal{O}}_{{\bf p}}$ with ${\bf p}=(2\mathbf{p}_{1} )_C=[2^n]$  by Proposition \ref{Annv}.
    
\end{enumerate}

When $n$ is odd, the argument is similar to the case when $n$ is even. We omit the details here.

This finishes the proof.
\end{proof}

\begin{Cor}
    Keep notations as above. When $V(L(\lambda))=\overline{\mathcal{O}}_{k(\lambda)}$, we will have 
$V(\Ann(L(\lambda)))=\overline{\mathcal{O}}_{{\bf p}}$ with
		\[	{\bf p}=
			[2^{k(\lambda)},1^{2n-2k(\lambda)}].
		\]		
\end{Cor}

\begin{Rem}
    From \cite[Lem. 3.17]{EHW}, we know that the generalized Verma module $N(\lambda)$ will be irreducible when $\lambda$ is not integral or half integral. Thus we will have $V(\Ann(L(\lambda)))=V(\Ann(N(\lambda)))=\overline{\mathcal{O}}_{{\bf p}}$ with ${\bf p}=[2^{n}]$, which is just the Richardson nilpotent  orbit  $G\cdot
    \mathfrak{p}^+$, see \cite{BZ}.
\end{Rem}

\section{Annihilator varieties  for type \texorpdfstring{$D_{n}$}{}}\label{ann-2d}
In this section, we consider  the case of type $D_n$.
\begin{Thm}\label{so(2,2n-2)}
		Let $L(\lambda)$ be a Harish-Chandra  module of $SO(2,2n-2)$ with  highest weight $\lambda-\rho=\lambda_0+z\xi\in\mathfrak{h}^*$. Denote $\lambda_0=(\lambda_1,\lambda_2,\cdots,\lambda_n)$, then $V(\Ann(L(\lambda)))=\overline{\mathcal{O}}_{{\bf p}}$ with
	\[	{\bf p}=\left\{
\begin{array}{ll}
	[1^{2n}], &
	\textnormal{~if~}z\in \lambda_2-\lambda_1+\mathbb{Z}_{\geq 0}, \\	
	\lbrack2^{2},1^{2n-4}\rbrack, & \textnormal{~if~}z+\lambda_1-\lambda_2\in\mathbb{Z}\textnormal{~and~}-|\lambda_n|-\lambda_1-n+1<z\le \lambda_2-\lambda_1-1,\\
	\lbrack 3,1^{2n-3}\rbrack, &
	\textnormal{~otherwise.}\\			
\end{array}	
\right.
\]				
\end{Thm}
\begin{proof}
	From \cite{EHW} we know that $\lambda-\rho=\lambda_0+z\xi$ with $\lambda_0=(\lambda_1,\lambda_2,\cdots,\lambda_n)$, $\lambda_2\geq\lambda_3\geq\cdots\lambda_{n-1}\geq|\lambda_n|$, $\lambda_i-\lambda_j\in\mathbb{N} \textnormal{~and~}2\lambda_i\in\mathbb{N}$ for $2\le i<j\le n$. By the normalization $( \lambda_0+\rho,\beta )=0$, we have $\lambda_1+\lambda_2=-2n+3$.

\begin{enumerate}[leftmargin=18pt]
\item When $\lambda $ is integral,  we have $\xi=(1,0,\dots,0)$ and 
	\begin{align*}
	  \lambda=(z+\lambda_1+n-1,\lambda_2+n-2,\cdots,\lambda_{n-1}+1,\lambda_n):=(t_1,...,t_n).
	\end{align*}
 Thus we have   $t_2> t_3>\cdots> t_{n-1}> |t_n|$.
 
 When $t_1>t_2$ (equivalently $z\in \lambda_2-\lambda_1+\mathbb{Z}_{\geq 0}$), $L(\lambda)$ will be finite-dimensional and $V(\Ann(L(\lambda)))=\overline{\mathcal{O}}_{{\bf p}}$ with the partition ${\bf p}=[1^{2n}]$ corresponding to the trivial orbit of type $D_n$.

When $t_1\le t_2$,  by using the Robinson-Schensted algorithm for $\lambda^-$,
  we can get  a Young tableau $P(\lambda^-)$ which consists of
  at most four columns. When $-|t_n|<t_1\le t_2$ (equivalently $-|\lambda_n|-\lambda_1-n+1<z\le \lambda_2-\lambda_1-1$), from the construction process, we can see that $P(\lambda^-)$ will be a Young tableau consisting of two columns with shape ${\bf p}(\lambda^-)=[2^2,1^{2n-4}]$ or  $[2^3,1^{2n-6}]$. 
 %$c_1(P(\lambda^-))=2n-3$ and $c_2(P(\lambda^-))=3$.
 In this case, the special partition of type $D$ corresponding to ${\bf p}(\lambda^-)$ is ${\bf p}=[2^2, 1^{2n-4}]$. Therefore $V(\Ann(L(\lambda)))=\overline{\mathcal{O}}_{{\bf p}}$ by Proposition \ref{Annv}.

When $t_1\leq -|t_n|$, $P(\lambda^-)$
 will be a Young tableau with shape ${\bf p}(\lambda^-)=[3,1^{2n-3}]$ or  ${\bf p}(\lambda^-)=[4,1^{2n-4}]$.  In this case, the special partition of type $D$ corresponding to ${\bf p}(\lambda^-)$ is  ${\bf p}=[3, 1^{2n-3}]$. Therefore $V(\Ann(L(\lambda)))=\overline{\mathcal{O}}_{{\bf p}}$ by Proposition \ref{Annv}.

\item When $\lambda$ is half integral, 
we have $\lambda_{Y_1}=(t_2,\cdots,t_n)\in [\lambda]_{1}$ and $\lambda_{Y_2}=(t_1)\in [\lambda]_2$ (when $z\in \frac{1}{2}+\mathbb{Z}$), or $\lambda_{Y_1}\in [\lambda]_{2}$ and $\lambda_{Y_2}=(t_1)\in [\lambda]_1$ (when $z\in \frac{1}{2}+\mathbb{Z}$). 
 Then we can get a Young  tableau $P(\lambda_{Y_1}^-)$ with shape ${\bf p}(\lambda_{Y_1}^-)=[1^{2n-2}]$ or $[2,1^{2n-4}]$, and a Young  tableau $P(\lambda_{Y_2}^-)$ with shape
\[	{\bf p}(\lambda_{Y_2}^-)=\left\{
		\begin{array}{ll}
			[2], & \textnormal{~if~}t_1\leq 0,\\	   	  \lbrack 1,1\rbrack,  & \textnormal{~if~}t_1> 0.
		\end{array}	
		\right.
		\]		

    \begin{enumerate}[leftmargin=18pt]
    \item If $t_1\in\frac{1}{2}+\mathbb{Z}$ (equivalently $z\in \frac{1}{2}+\mathbb{Z}$), we will have $\lambda_{Y_1}\in [\lambda]_{1}$ and $ \lambda_{Y_2}\in [\lambda]_{2}$. From Proposition \ref{Annv}, we have $\mathbf{p}_{0}=[1^{2n-2}]$ and 
    the $C$-type metaplectic
special partition ${\mathbf p}_{\frac{1}{2}}$ corresponding to $p(\lambda_{Y_2}^-)$ is 
${\mathbf p}_{\frac{1}{2}}=[2]$. 
    Therefore $V(\Ann(L(\lambda)))=\overline{\mathcal{O}}_{{\bf p}}$ with  ${\bf p}=(\mathbf{p}_{0} {\stackrel{c}{\sqcup}} {\mathbf p}_{\frac{1}{2}})_B=[3,1^{2n-3}]$ by Proposition \ref{Annv}.

    \item If $t_1\in\mathbb{Z}$  (equivalently $z\in \frac{1}{2}+\mathbb{Z}$),  we will have $\lambda_{Y_1}\in [\lambda]_{2}$ and $ \lambda_{Y_2}\in [\lambda]_{1}$. From Proposition \ref{Annv},  the $C$-type metaplectic
special partition  corresponding to $p(\lambda_{Y_1}^-)$ is 
$\mathbf{p}_{\frac{1}{2}}=[2,1^{2n-4}]$ and $\mathbf{p}_{0}=[1,1]$ if $t_1>0$ or $\mathbf{p}_{0}=[2]$ if $t_1\leq 0$. Therefore $V(\Ann(L(\lambda)))=\overline{\mathcal{O}}_{{\bf p}}$ with 
    $	{\bf p}=(\mathbf{p}_{0} {\stackrel{c}{\sqcup}} {\mathbf p}_{\frac{1}{2}})_B	=\lbrack 3,1^{2n-3}\rbrack
		$ by Proposition \ref{Annv}.
    \end{enumerate}    
   \item When $\lambda$ is not integral or half integral, 
we have $\lambda_{Y_1}=(t_2,\cdots,t_n)\in [\lambda]_{1,2}$  and $\lambda_{Y_2}=(t_1)\in [\lambda]_3$ (equivalently $z\notin \frac{1}{2}\mathbb{Z}$). 
 Then we can get a Young  tableau $P(\lambda_{Y_1}^-)$ with shape ${\bf p}(\lambda_{Y_1}^-)=[1^{2n-2}]$ or $[2,1^{2n-4}]$, and a Young  tableau $P(\lambda_{Y_2})$ with shape $[1]$.

                   \begin{enumerate}[leftmargin=18pt]
    \item If  $\lambda_{Y_1}\in [\lambda]_{1}$, we will have $\mathbf{p}_{0}=[1^{2n-2}]$. From $P(\lambda_{Y_2})$, we can get a partition ${\bf p}_1=[1]$.
 Therefore $V(\Ann(L(\lambda)))=\overline{\mathcal{O}}_{{\bf p}}$ with  ${\bf p}=(\mathbf{p}_{0} {\stackrel{c}{\sqcup}} {2\mathbf p}_{1})_B=[3,1^{2n-3}]$ by Proposition \ref{Annv}.

    \item If  $\lambda_{Y_1}\in [\lambda]_{2}$, the $C$-type metaplectic
special partition  corresponding to ${\bf p}(\lambda_{Y_1}^-)$ is $\mathbf{p}_{\frac{1}{2}}=[2,1^{2n-4}]$. From $P(\lambda_{Y_2})$, we can get a partition ${\bf p}_1=[1]$.
    Therefore $V(\Ann(L(\lambda)))=\overline{\mathcal{O}}_{{\bf p}}$ with  ${\bf p}=(\mathbf{p}_{\frac{1}{2}} {\stackrel{c}{\sqcup}} {2\mathbf p}_{1})_B=[3,1^{2n-3}]$ by Proposition \ref{Annv}.

    % $d=(\mathbf{p}_{0} {\stackrel{c}{\sqcup}} {\mathbf p}_{\frac{1}{2}})_B=[2^2,1^{2n-4}]$ if $t_1>0$ or $d=(\mathbf{p}_{0} {\stackrel{c}{\sqcup}} {\mathbf p}_{\frac{1}{2}})_B=[3,1^{2n-3}]$ if $t_1\leq 0$  by Proposition \ref{Annv}.
\end{enumerate}

\end{enumerate}

This finishes the proof.

\end{proof}

\begin{Rem}
    From \cite[Lem. 3.17]{EHW}, we know that the generalized Verma module $N(\lambda)$ will be irreducible when $\lambda$ is not integral. Thus we will have $V(\Ann(L(\lambda)))=V(\Ann(N(\lambda)))=\overline{\mathcal{O}}_{{\bf p}}$ with ${\bf p}=[3,1^{2n-3}]$, which is just the Richardson nilpotent  orbit  $G\cdot
    \mathfrak{p}^+$, see \cite{BZ}.
\end{Rem}

% \begin{Def}
% 	Let $d$ be a partition of $2n$ in type $D_n$, if there are only even parts, and each part has even multiplicity, $d$ is called a {\bf very even partition}.
% \end{Def}
% \begin{Exam}
% 	In $\mathfrak{g}=\mathfrak{so}_8$, $d=[4^2]\textnormal{~and~}[2^4]$ are very even partitions. Thus, we get four nilpotent orbits:
% 	$$\overline{\caO}_{[4^2]}^{\uppercase\expandafter{\romannumeral1}},\:\overline{\caO}_{[4^2]}^{\uppercase\expandafter{\romannumeral2}},\:\overline{\caO}_{[2^4]}^{\uppercase\expandafter{\romannumeral1}},\:\overline{\caO}_{[2^4]}^{\uppercase\expandafter{\romannumeral2}}.$$
% \end{Exam}

\begin{Thm}\label{D-AV}
		Let $L(\lambda)$ be a Harish-Chandra module of $SO^{\ast}(2n)$ with  highest weight $\lambda-\rho=\lambda_0+z\xi\in\mathfrak{h}^*$. Suppose $q_2$ is  the length of second column of the Young
tableau $P(\lambda^-)$.
Denote $\lambda_0=(\lambda_1,\lambda_2,\cdots,\lambda_n)$, then $V(\Ann(L(\lambda)))=\overline{\mathcal{O}}_{{\bf p}}$ with	
			\[	{\bf p}=\left\{
		\begin{array}{ll}
			[2^{2q_2^{\ev}},1^{2n-4q_2^{\ev}}], &
			\textnormal{~if~}z\in\mathbb{Z}, \\				\lbrack 2^{n}\rbrack, &
			\textnormal{~if~}n\in 2\mathbb{Z}, \\			
			\lbrack 2^{n-1},1,1\rbrack, &
			\textnormal{~if~}n\in 2\mathbb{Z}+1. 				
		\end{array}	
		\right.
		\]		
\end{Thm}
\begin{proof}
From \cite{EHW} we know that $\lambda-\rho=\lambda_0+z\xi$ with $\lambda_0=(\lambda_1,\lambda_2,\cdots,\lambda_n)$, and $\lambda_i-\lambda_j\in\mathbb{N}$ for $i\le j$. By the normalization $( \lambda_0+\rho,\beta )=0$, we have $\lambda_1+\lambda_2=-2n+3$ and $2\lambda_i\in\mathbb{Z}$ for $1\leq i\leq n$.	

% Suppose that $\lambda_i\in\mathbb{Z}\textnormal{~for~}1\le i\le n$.

\begin{enumerate}[leftmargin=18pt]
    \item When $\lambda$ is integral (equivalently $z\in\mathbb{Z}$), we have $\xi=(\underbrace{\frac{1}{2},\frac{1}{2},\dots,\frac{1}{2}}_{n})$ and $$\lambda=(\frac{1}{2}z+\lambda_1+n-1,\frac{1}{2}z+\lambda_2+n-2,\cdots,\frac{1}{2}z+\lambda_{n-1}+1,\frac{1}{2}z+\lambda_n):=(t_1,t_2,\cdots,t_n).$$
Thus we have   $t_1> t_2>\cdots> t_{n}$.

% When $t_n\geq 0$, $L(\lambda)$ will be finite-dimensional and $V(\Ann(L(\lambda)))=\overline{\mathcal{O}}_{d}$ with the partition $d=[1^{2n}]$ corresponding to the trivial orbit of type $D_n$. 
    
By using the Robinson-Schensted algorithm for $\lambda^-$,
  we can get  a Young tableau $P(\lambda^-)$ which consists of
at most two columns with $c_1(P(\lambda^-))=2n-q_2$ and $c_2(P(\lambda^-))=q_2$. Thus ${\bf p}(\lambda^-)=[2^{q_2},1^{2n-2q_2}]$. In this case the special partition of type $D$ corresponding to ${\bf p}(\lambda^-)$ is ${\bf p}=[2^{2q_2^{\ev}},1^{2n-4q_2^{\ev}}]$ by  (\ref{eq:ev-od}). Therefore $V(\Ann(L(\lambda)))=\overline{\mathcal{O}}_{{\bf p}}$ by Proposition \ref{Annv}.

\item When $\lambda$ is not integral (equivalently $z\notin\mathbb{Z}$), we will have $\lambda\in[\lambda]_3$. From Proposition \ref{Annv} we have ${\bf p}_1=p(\lambda_{Y_1})=[1^n]$.  Therefore $V(\Ann(L(\lambda)))=\overline{\mathcal{O}}_{{\bf p}}$ with $${\bf p}=( 2\mathbf{p}_{1} )_D=\left\{
		\begin{array}{ll}
			[2^n], &
			\textnormal{~if~}n\in 2\mathbb{Z}, \\			
			\lbrack 2^{n-1},1,1\rbrack, &
			\textnormal{~if~}n\in 2\mathbb{Z}+1. 					\end{array}	
		\right.
$$  by Proposition \ref{Annv}.  
\end{enumerate}
\end{proof}

\begin{Cor}
    Keep notations as above. When $V(L(\lambda))=\overline{\mathcal{O}}_{k(\lambda)}$, we will have 
$V(\Ann(L(\lambda)))=\overline{\mathcal{O}}_{{\bf p}}$ with
		\[	{\bf p}=
			[2^{2k(\lambda)},1^{2n-4k(\lambda)}].
		\]		
\end{Cor}

\begin{Rem}
	 From \cite[Lem. 3.17]{EHW}, we know that the generalized Verma module $N(\lambda)$ will be irreducible when $\lambda$ is not integral. Thus when $n$ is even, we will have $$V(\Ann(L(\lambda)))=V(\Ann(N(\lambda)))=\overline{\mathcal{O}}_{{\bf p}}$$ with ${\bf p}=[2^{n}]$, which is just the Richardson nilpotent  orbit  $G\cdot
    \mathfrak{p}^+$, see \cite{BZ}.
 
 When $G=SO^{\ast}(2n)$, from Proposition \ref{bcd}, we know that each very even partition ${\bf p}$ corresponds to two special nilpotent orbits $\overline{\caO}_{{\bf p}}^{\uppercase\expandafter{\romannumeral1}}$ and $\overline{\caO}_{{\bf p}}^{\uppercase\expandafter{\romannumeral2}}$. 
 Note that in $SO^{\ast}(2n)$,  ${\bf p}=[2^n]$ is a very even partition when $n$ is even, which  corresponds to two nilpotent orbits. In fact, the Levi subalgebra $\mathfrak{k}$ of our parabolic subalgebra $\mathfrak{q}=\mathfrak{k}\oplus\mathfrak{p}^+$ is 
 of type $I$
 by \cite[Lem. 7.3.2]{CM}. From \cite[Lem. 7.3.3]{CM}, we know that the numeral of $\mathcal{O}$ appeared in $V(\Ann(L(\lambda)))$ is $I$ if $n$ is even.
  In other words, we have $$V(\Ann(L(\lambda)))=      \overline{\caO}_{[2^n]}^{\uppercase\expandafter{I}}\text{~when~ $n$ ~is~ even}.
 $$
\end{Rem}

\section{Annihilator varieties  for types \texorpdfstring{$E_{6}$}{} and \texorpdfstring{$E_{7}$}{} }\label{ann-2e}
In this section,  we consider annihilator varieties  of highest weight Harish-Chandra modules for two exceptional groups of type $E_{6(-14)}$ and $E_{7(-25)}$. We recall some notations from \cite{Ca85} and \cite{CM}.  

Let $\mathfrak{g}$ be a semisimple Lie algebra. Let $\mathfrak{l}$ be a Levi subalgebra of $\mathfrak{g}$ and  
 $\mathfrak{p}_{\mathfrak{l}}$ be a distinguished parabolic subalgebra of the semisimple algebra $[\mathfrak{l},\mathfrak{l}]$.
 From  \cite[Thm. 8.2.12]{CM}, we know that 
 a nilpotent orbit is corresponding to an ordered pair $(\mathfrak{l},\mathfrak{p}_{\mathfrak{l}})$.
 From \cite{BC}, we denote this orbit by the label $X_N(a_i)$, where $X_N$ is the Cartan type of the semisimple part of $\mathfrak{l}$ and $i$ is the number of simple roots in any Levi subalgebra of $\mathfrak{p}_{\mathfrak{l}}$. If $i=0$, one writes $X_N$ rather than $X_N(a_0)$.  When $\mathfrak{g}$ is of type $E_7$,   it  has two  non-conjugate isomorphic Levi subalgebras. One subalgebra  is chosen arbitrarily and labeled with a prime, and the other one has a double prime. 

In the following tables we give four pieces of information about the nilpotent orbits which will be used: Bala-Carter label, dimension, fundamental group and whether the orbit is special or not. The Bala-Carter label $\mathcal{K}=sX_N$ denotes $s$ copies of $X_N$.
                                                                
% \begin{align*}\label{e6}
% 	&\small{\textbf{ Table~2 }  \text { Some nilpotent orbits in type $E_6$} }\\
% 	&\begin{array}{ccccccc}	
% 		\toprule[1pt] 
% 		{\rm label}~\mathcal{K}&~& \dim\mathcal{O} &~& \pi_1(\mathcal{O}) &~~ &{\rm special~orbit}\\
% 		\midrule[0.6pt]		
% 		 0 &~&0 & ~& 1 &~~ &{\rm yes}\\
% 		 A_1 &~& 22 &~ & 1 &~~ &{\rm yes}\\
% 		 2A_1&~& 32 &~ & 1 & ~~&{\rm yes}\\
% 		 3A_1&~ & 40 &~ & 1 & ~~&{\rm no}\\
% 		 A_2& ~& 42 & ~& S_2 & ~~&{\rm yes}\\
% 		\bottomrule[1pt]
% 	\end{array}
% \end{align*}

\begin{table}[htp]
\centering
\renewcommand{\arraystretch}{1.5}
\setlength\tabcolsep{10pt}
\caption{Some nilpotent orbits in type $E_6$}\label{e6} 
\begin{tabular}{cccc}
\toprule[1pt] 
${\rm label}~\mathcal{K}$& $\dim\mathcal{O}$ & $\pi_1(\mathcal{O})$  &${\rm special~orbit}$\\
\hline 
$0$ &$0$ &  $1$  &{\rm yes}\\
		 $A_1 $& $22$ & $1$  &${\rm yes}$\\
		 $2A_1$& $32 $ & $1$ & ${\rm yes}$\\
		 $3A_1$ & $40$  & $1$ & ${\rm no}$\\
		 $A_2$&  $42$ & $S_2 $& ${\rm yes}$\\
\bottomrule[1pt]
\end{tabular}

\end{table}

% \end{minipage}
% \begin{minipage}{0.6\textwidth}
% \begin{align*}
% 	&\small{\textbf{ Table~3 }  \text { Some nilpotent orbits in type $E_7$} }\\
% 	&\begin{array}{ccccccc}	
% 		\toprule[1pt] 
% 		{\rm label}~\mathcal{K}&~& \dim\mathcal{O} &~& \pi_1(\mathcal{O}) &~~ &{\rm special~orbit}\\
% 		\midrule[0.6pt]		
% 		 0 &~&0 & ~& 1 &~~ &{\rm yes}\\
% 		 A_1 &~& 34 &~ & 1 &~~ &{\rm yes}\\
% 		 2A_1&~& 52 &~ & 1 & ~~&{\rm yes}\\
% 		 (3A_1)''&~ & 54 &~ & \mathbb{Z}/2\mathbb{Z} & ~~&{\rm yes}\\
% 		 (3A_1)'& ~& 64 & ~& 1 & ~~&{\rm no}\\
% 		 A_2 &~& 66 &~ & S_2 &~~ &{\rm yes}\\
% 		\bottomrule[1pt]
% 	\end{array}
% \end{align*}
% \end{minipage}
\begin{table}[htp]
\centering
\renewcommand{\arraystretch}{1.5}
\setlength\tabcolsep{10pt}
\caption{Some nilpotent orbits in type $E_7$}\label{e7} 
\begin{tabular}{cccc}
\toprule[1pt] 
${\rm label}~\mathcal{K}$& $\dim\mathcal{O}$ & $\pi_1(\mathcal{O})$  &${\rm special~orbit}$\\
\hline 
$0$ &$0$ &  $1$  &{\rm yes}\\
			 $A_1$ & $34$  & $1$  &${\rm yes}$\\
		$ 2A_1$& $52$  & $1$ & ${\rm yes}$\\
		 $(3A_1)''$ & $54$  &$\mathbb{Z}/2\mathbb{Z}$ & ${\rm yes}$\\
		 $(3A_1)'$& $ 64$ &  $1$ & ${\rm no}$\\
		 $A_2$ &$ 66$  & $S_2$  &${\rm yes}$\\
\bottomrule[1pt]
\end{tabular}

\end{table}

\begin{Thm}\label{e6-GK}
	Let $L(\lambda)$ be a highest weight Harish-Chandra module of the exceptional Lie group $E_{6(-14)}$ with highest weight $\lambda-\rho=\lambda_0+z\xi\in\mathfrak{h}^*$ and $V(L(\lambda))=\overline{\caO}_{k(\lambda)}$. Denote $\lambda_0=(\lambda_1,\lambda_2,\cdots,\lambda_8)$, then
	$V(\Ann(L(\lambda)))=\overline{\mathcal{O}}_{\mathcal{K}}$ with	
				\[	\mathcal{K}=\left\{
			\begin{array}{ll}
			\vspace{1ex}
	0, & \textnormal{~if~}z>-(\lambda_0, \alpha_1)-1 \textnormal{~and~}\lambda \textnormal{~is~integral},\\	\vspace{1ex}	
		A_1, &\textnormal{~if~}-4+\min\{-(\lambda_0,\beta_{1,2})\}<z\le-(\lambda_0,\alpha_1)-1 \textnormal{~and~}\lambda \textnormal{~is~integral},\\		\vspace{1ex}
		2A_1, & 	\textnormal{~otherwise.}\\				
			\end{array}	
			\right.
			\]	
Here $\beta_{1,2}$ means the root $\beta_1$ and $\beta_2$ in Proposition \ref{E}	(1).				
\end{Thm}
\begin{proof}
From \cite{EHW} we know that $\lambda-\rho=\lambda_0+z\xi$ with $\lambda_0=(\lambda_1,\lambda_2,\cdots,\lambda_8)$, $|\lambda_1|\le\lambda_2\le\cdots\lambda_5$, $\lambda_i-\lambda_j\in\mathbb{Z}$ and $2\lambda_i\in\mathbb{Z}$ for $i, j\le5$. By the normalization $(\lambda_0+\rho,\beta)=0$, we have $(\lambda_0,\beta)=-11$.

\begin{enumerate}[leftmargin=18pt]
\item When $\lambda$ is integral, we have $\xi=(0,0,0,0,0,-\frac{2}{3},-\frac{2}{3},\frac{2}{3})$
\begin{enumerate}[leftmargin=18pt]
    \item When $\Psi_{\lambda}^+\cap S_2\ne\emptyset$, equivalently $(\lambda, \alpha_1^{\vee})>0$, we will have
$k(\lambda)=0$ by Proposition \ref{E}.
By Corollary \ref{dim} and Table \ref{e6}, we can get that the label of the corresponding nilpotent orbit is  $\mathcal{K}=0$.    

Hence  $(\lambda, \alpha_1^{\vee})>0$ if and only if 
$z>-(\lambda_0, \alpha_1)-1$. In this case the label $\mathcal{K}=0$.
    
    \item When $\Psi_{\lambda}^+\cap S_1=\emptyset$, equivalently, $(\lambda, \beta_i^{\vee})\leq 0$ for $i=1,2$, we will have
$k(\lambda)=2$ by Proposition \ref{E}.
By Corollary \ref{dim} and Table \ref{e6}, we can get that the label of the corresponding nilpotent orbit is  $\mathcal{K}=2A_1$.

Hence  $(\lambda, \beta_i^{\vee})\leq 0$ if and only if 
$z\le -4+\min\{-(\lambda_0, \beta_i)\}$ for $i=1,2$. In this case the label $\mathcal{K}=2A_1$.

\item When  $\Psi_{\lambda}^+\cap S_1\ne\emptyset$ and $\Psi_{\lambda}^+\cap S_2=\emptyset$, equivalently, $z>-4+\min\{-(\lambda_0, \beta_i)\}$ and $z\le -(\lambda_0, \alpha_1)-1$, we will have
$k(\lambda)=2$ by Proposition \ref{E}.
By Corollary \ref{dim} and Table \ref{e6}, we can get that the label of the corresponding nilpotent orbit is  $\mathcal{K}=A_1$.

Hence   when $-4+\min\{-(\lambda_0, \beta_i)\}<z\le -(\lambda_0, \alpha_1)-1$, the label $\mathcal{K}=A_1$.

\end{enumerate}

\item When $\lambda$ is non-integral, we will have
$k(\lambda)=2$ by Proposition \ref{E}.
By Corollary \ref{dim} and Table \ref{e6}, we can get that the label of the corresponding nilpotent orbit is $\mathcal{K}=2A_1$.
\end{enumerate}

\end{proof}

\begin{Thm}\label{e7-GK}
	Let $L(\lambda)$ be a highest weight Harish-Chandra module of the exceptional Lie group $E_{7(-25)}$ with highest weight $\lambda-\rho=\lambda_0+z\xi\in\mathfrak{h}^*$ and $V(L(\lambda))=\overline{\caO}_{k(\lambda)}$. Denote $\lambda_0=(\lambda_1,\lambda_2,\cdots,\lambda_8)$, then
	$V(\Ann(L(\lambda)))=\overline{\mathcal{O}}_{\mathcal{K}}$ with	
				\[	\mathcal{K}=\left\{
			\begin{array}{ll}
			\vspace{1ex}
	0, & \textnormal{if~}z>-(\lambda_0,\beta_8)-1\textnormal{~and~}\lambda \textnormal{~is~integral},\\	\vspace{1ex}	
		A_1, &\textnormal{if~}-(\lambda_0,\beta_{6,7})-5< z\le-(\lambda_0,\beta_8)-1\textnormal{~and~}\lambda \textnormal{~is~integral},\\		\vspace{1ex}
		2A_1, & \textnormal{if~}  -9+\min\{-(\lambda_0,\beta_{3,4,5})\} <z\le -(\lambda_0,\beta_{6,7})-5\textnormal{~and~}\lambda \textnormal{~is~integral},\\	
		(3A_1)'', & 	\textnormal{otherwise.}\\		
			\end{array}	
			\right.
			\]	
Here $\beta_{3,4,5}$	means the root $\beta_3$, $\beta_4$ and $\beta_5$ in Proposition \ref{E}	(2). Similarly $\beta_{6,7}$ means the root  $\beta_6$ and $\beta_7$ in Proposition \ref{E}	(2).
\end{Thm}
\begin{proof}
The argument is similar to the case of $E_{6(-14)}$. So we omit the details here.
\end{proof}

\section{The Gelfand-Kirillov dimension of a unitary highest weight module}\label{GKdim-k}
Proposition \ref{unitary} indicates that if $L(\lambda)$ is a unitary highest weight module, then ${\rm GKdim}\:L(\lambda)$ only depends on the value $z=(\lambda,\beta^{\vee})$. The proof of Proposition \ref{unitary} used some results from Joseph \cite{Jo92}. In this section, we will prove this result by  another  method. 

\begin{Lem}[{\cite[Lem. 5.4]{BHXZ}}]\label{antichain}
	Suppose $\mu=(t_1,\cdots,t_n,s_1,\cdots,s_m)$ is a $(n,m)$-dominant sequence with $t_n\le s_1$. Then applying R-S algorithm to $\mu$, we can get a Young tableau $P(\mu)$ with two columns, and the number of boxes in the second column  is precisely the largest integer $k$ for which we have 
	$$t_{n-k+1}\le s_1,\:t_{n-k+2}\le s_2,\:\cdots,\: t_{n-1}\le s_{k-1},\:t_n\le s_{k}.
	$$
\end{Lem}

% When $G=SO(2,2n-1)$ and $SO(2,2n-2)$, Proposition \ref{BCD} shows that no matter how much $\lambda_0$ changes, $k(\lambda)$ is always a constant. This can indicate that ${\rm GKdim}\:L(\lambda)$ is independent of the selection of $\lambda_0$, next we only consider the case when $G=SU(p,n-p)$, $Sp(n,\mathbb{R})$ and $SO^{\ast}(2n)$.

Now we recall two root systems
 $Q(\lambda_{0})$ and $ R(\lambda_{0})$ given in \cite{EHW}. Their constructions are as follows.  Let $\Phi_{c}(\lambda_{0})=\{\alpha \in \Phi(\mathfrak{k})| (\lambda_{0}, \alpha)=0\}$. Consider the subroot system $\Psi_1$ of $\Phi$, which is generated by $\pm \beta$ and $\Phi_{c}(\lambda_{0})$. Let  $Q(\lambda_{0})$ be the simple component of  $\Psi_1$ which contains $-\beta.$ If $\Phi$ has two root lengths  and there exist short  roots $\alpha'\in \Phi(\mathfrak{k})$ which are not orthogonal to $Q(\lambda_{0})$ and satisfy $(\lambda_{0},\alpha'^{\vee})=1$, then let $\Psi_2$ be the root system generated by $\pm\beta$, $\Phi_{c}(\lambda_{0})$ and all such $\alpha$.
Let $R(\lambda_{0})$ be the simple component of $\Psi_2$ which contains $-\beta$. If $\Phi$ has only one root length or no such $\alpha$ exists,  we let $R(\lambda_{0})=Q(\lambda_{0})$.

\begin{Thm}\label{kconst}
	When $L(\lambda)$ is a unitary highest weight module and $z=z_k=(\rho,\beta^{\vee})-kc$, we will have 
$V(L(\lambda))=\overline{\mathcal{O}}_{k(\lambda)}$, where $k(\lambda)=r$ if $k\geq r$ and $k(\lambda)=k$
if $0\leq k\leq r-1$.

 % $k(\lambda)=-(\lambda-\rho,\beta^{\vee})/{c}$.
\end{Thm}

	When $L(\lambda)$ is a unitary highest weight module, we will have $V(L(\lambda))=\overline{\caO}_{k(\lambda)}$ with $k(\lambda)$  given in Proposition \ref{BCD}. From \cite{EHW}, we can write $\lambda-\rho=\lambda_0+z\xi$, where $\xi$ is the fundamental weight perpendicular to the compact roots such that $(\xi,\beta^{\vee})=1$ and $\lambda_0=(\lambda_{1},\lambda_2,\cdots,\lambda_{n})$ with $(\lambda_0+\rho,\beta^{\vee})=0$.
	
Now   $z=z(\lambda)=(\lambda,\beta^{\vee})=z_k=(\rho,\beta^{\vee})-kc$, we want to prove that $k(\lambda)=r$ if $k\geq r$ and $k(\lambda)=k$ if $0\leq k\leq r-1$.

In the following, we will give the proof of our Theorem \ref{kconst} in a case-by-case way. The idea of our proof is uniform. For $G=SU(p,q),Sp(n,\mathbb{R})$ and $SO^*(2n)$, we use Lemma \ref{antichain} to prove our result. For other cases, we use our characterizations of annihilator varieties of $L(\lambda)$ since they are given in a distribution of the value of $z=z(\lambda)=(\lambda,\beta^{\vee})$.

\subsection{Proof for $G=SU(p,q)$} 
For $SU(p,q)$ with $p+q=n$,  from \cite{EHW} the root system $Q(\lambda_0)=R(\lambda_0)$ is of type $\mathfrak{su}(p',q')$ with $p'\leq p$ and $q'\leq q$. Then we have $\lambda_1=\lambda_2=\cdots=\lambda_{p'}>\lambda_{p'+1}$
and $\lambda_n=\lambda_{n-1}=\cdots=\lambda_{n-q'+1}<\lambda_{n-q'}$. 

From \cite{EHW} we know that $L(\lambda)$ is unitarizable if and only if $z\leq \max\{p', q'\}$ or $z\in \mathbb{Z}$  with $z\leq p'+q'-1$. In this case, we have $c=1$, $2\rho=(n-1, n-3, ...,-n+3, -n+1)$, $(\rho, \beta^{\vee})=n-1$, $\lambda_1-\lambda_n+n-1=0$ and  $n\xi=(q,...,q,-p,...,-p)$ with $p$ copies of $q$ and $q$ copies of $-p$.

We define an equivalent relation, such that $\lambda \sim \lambda'$ if and only if  $q_2(\lambda) =q_2( \lambda')$ where $q_2(\lambda)$ denotes the number of boxes in the second column of the Young tableau $P(\lambda)$. We denote $e=(1,1,...,1)$ with $n$ copies of $1$ and $e_0=(1,...,1,0,...,0)$ with $p$ copies of $1$ and $q$ copies of $0$.
By using R-S algorithm, we have $$\lambda=\lambda_0+z\xi+\rho \sim \lambda_0+z\xi+\rho+\frac{zp}{n}e= \lambda_0+\rho+ze_0\sim \lambda_0+\rho+ze_0 +\frac{n-1}{2}e.$$
Thus  \begin{align*}
&\lambda\sim \lambda_0+\rho+ze_0 +\frac{n-1}{2}e\\
&=(\lambda_n+z,\lambda_n+z-1,...,\lambda_n+z-p'+1, \lambda_{p'+1}+z+n-p'-1,...,\lambda_p+z+n-p,\\
&\quad\quad \lambda_{p+1}+n-p-1,...,\lambda_{n-q'}+q',\lambda_n+q'-1,...,\lambda_n+1, \lambda_n)\\
&:=(s_1,...,s_p,s_{p+1},...,s_{n}).
\end{align*}

We may assume $p\leq q$. Then $(s_1,...,s_p)$ is a strictly decreasing sequence smaller that  $\lambda_n+z$ and $(s_{p+1},...,s_n)$ is a strictly decreasing sequence larger than $\lambda_n$.

Thus we have 
\[	
\left\{
\begin{array}{ll}     	  	
	s_{p+m}\geq s_n+q-m, &{\rm ~if~}\:1\leq m\leq q\\	
	s_i\leq s_n+z-i+1, &{\rm ~if~}\:1\leq i\leq p.
\end{array}	
\right.
\]	
% $s_{p+m}\geq s_n+q-m$ for $1\leq m\leq q$ and $s_i\leq s_n+z-i+1$ for $1\leq i\leq p$.

Note that $z_k=(\rho,\beta^{\vee})-kc=n-1-k$, so we have 
$$\{z_k\mid 0\le k\le r=p\}=\{n-1,n-2,\dots, q,q-1\}.$$

Now suppose that $L(\lambda)$ is a unitary highest weight module, then we have the follows.
	\begin{enumerate}[leftmargin=18pt]
\item When $p'+q'-1\leq q-1$,   we will have $z=z_k\leq p'+q'-1\leq q-1=z_r$. So we can get
\[	
\left\{
\begin{array}{ll}     	  	
	 s_1=s_n+z\leq s_n+q-1\leq  s_{p+1},\\	
	s_2\leq s_{p+2},\\
\quad\quad \vdots\\
s_{p'} \leq s_n+z+1-p'\leq  s_n+q-p' \leq s_{p+p'},\\
 s_{p'+1}\leq s_n+z-(p'+1)+1=s_n+z-p'\leq s_n+q-1-p'\leq s_{p+p'+1},\\
\quad\quad \vdots\\
 s_p\leq s_{2p}.
\end{array}	
\right.
\]	
  Thus we have $k(\lambda)=q_2(\lambda)\geq p$ by Lemma \ref{antichain} and Proposition \ref{BCD}. Since $k(\lambda)\leq r=\min\{p,q\}=p$, we must have $k(\lambda)=p=k$.
  \item 
When  $p'+q'-1\geq q$ and $q\leq z=z_k=n-1-k\leq  p'+q'-1$ for some $1\leq k\leq r-1=p-1 $, similarly we can get
\[	
\left\{
\begin{array}{ll}     	  	
	s_{p+m}\geq s_n+q-m,&1\leq m\leq q\\	
	s_i\leq s_n+z-i+1,&1\leq i\leq p.
\end{array}	
\right.
\]

% we can have $z(\lambda)\geq q$ for a unitary highest weight module $L(\lambda)$. Similarly we can show that  $k(\lambda)=p$ when $z\leq q-1$. 
% Now we suppose that $q\leq z\leq  p'+q'-1$ and $z=z_k=n-1-k$ for some $1\leq k\leq r-1=p-1 $. Similarly we have $s_{p+m}\geq s_n+q-m$ for $1\leq m\leq q$ and $s_i\leq s_n+z-i+1$ for $1\leq i\leq p$.

When $z-q+2=p-k+1\leq i\leq p$, we have $q-z+i-1=k+i-p\geq 1$ and
$$s_i\leq s_n+z-i+1=s_n+q-(q-z+i-1)\leq s_{p+q-z+i-1}=s_{k+i}.$$
So 
\[	
\left\{
\begin{array}{ll}     	  	
	s_{p-k+1}\leq s_{p+1},\\	
	\quad \quad \vdots\\
 s_{p}\leq s_{p+q-z+p-1}=s_{p+n-z-1}=s_{p+k}.
\end{array}	
\right.
\]	
% $s_{z-q+2}\leq s_{p+1}$,...,$s_{p}\leq s_{p+q-z+p-1}=s_{p+n-z-1}=s_{p+k}$. 
Thus we have $k(\lambda)=q_2(\lambda)\geq k$ by Lemma \ref{antichain}. 

If $q_2(\lambda)>k$, we will have $s_{p-k}\leq s_{p+1}, \cdots, s_{p}\leq s_{p+k+1}$. From $ z=z_k=n-1-k\leq  p'+q'-1$, we can get $$n-k\leq  p'+q'\Rightarrow q-q'+p-k\leq p'.$$
% Thus we can get
% $$p-k\leq p' \text{~and~} p+1>n-q'$$
Now we choose a  positive integer $1\leq i_0\leq k$ such that $$i_0\geq q-q' ~(\Leftrightarrow p+1+i_0> n-q')\text{~and~} p-k+i_0\leq p'. $$
Then we have $$s_{p-k+i_0}\leq s_{p+1+i_0}=\lambda_n+n-(p+1+i_0)=\lambda_n+n-p-1-i_0.$$
On the other hand, we have $$s_{p-k+i_0}=\lambda_n+z-(p-k+i_0)+1=\lambda_n+n-p-i_0.
$$
Therefore we get contradiction! So we must have $k(\lambda)= q_2(\lambda)=k$.

\item When  $p'+q'-1\geq q$ and $z=z_k=n-1-k\leq  q-1$, 
by similar arguments as in (1) we can show that  $k(\lambda)=p=k$.

\end{enumerate}

\subsection{Proof for $G=Sp(n,\mathbb{R})$} 
For $Sp(n,\mathbb{R})$, from \cite{EHW} the root system $Q(\lambda_0)$ is of type $\mathfrak{sp}(q,\mathbb{R})$ and $R(\lambda_0)$ is of type $\mathfrak{sp}(r,\mathbb{R})$ with $r\geq q$. Then from \cite{EHW} we have $\lambda_1=\cdots=\lambda_q=-n> \lambda_{q+1}\geq \cdots \geq \lambda_n$  and $\lambda_{q+1}=\cdots=\lambda_{r}=\lambda_1-1=-n-1>\lambda_{r+1}$. From Table \ref{constants}, we have $c=\frac{1}{2}$ and $(\rho,\beta^{\vee})=n$.
	
From \cite{EHW} we know that $L(\lambda)$ is unitarizable if and only if $z\le \frac{1}{2}(r+1)$	or $2z\in\mathbb{Z}$ with $z\le \frac{1}{2}(q+r)$.  Thus
\begin{align*}
\lambda=\lambda_0+z\xi+\rho=&(z,z-1,\cdots,z-q+1,z-q-1,\cdots,z-r,\\
&z+\lambda_{r+1}+n-r,\cdots,z+\lambda_n+1)\\
:=&(s_1,s_2,\cdots,s_q,s_{q+1},\cdots,s_r,s_{r+1},\cdots,s_n),
\end{align*}
then $(s_1,s_2,\cdots,s_n)$ and $(-s_n,-s_{n-1},\cdots,-s_1)$ are strictly decreasing sequences.

Note that $z_k=(\rho,\beta^{\vee})-kc=n-\frac{k}{2}$, so we have 
$$\{z_k\mid 0\le k\le r=n\}=\left\{n,n-\frac{1}{2},\cdots,\frac{n}{2}+\frac{1}{2},\frac{n}{2}\right\}.$$

Now suppose that $L(\lambda)$ is a unitary highest weight module, then we have the follows.
	\begin{enumerate}[leftmargin=18pt]
	\renewcommand{\labelenumi}{(\theenumi)}	
\item When $n$ is even and $\frac{1}{2}(q+r)\le\frac{1}{2}(n-1), \text{~in ~other ~words}, z\le\frac{1}{2}(n-1)=z_{n+1}$, we will have
\[	
	\left\{
\begin{array}{ll}
	s_m\le -s_1+n-m, &{\rm ~if~}\:1\le m\le q,\\	     	  	
s_l\le -s_1+2z-l+1, &{\rm ~if~}\:q+1\le l\le n,\\
-s_{n-i}\geq -s_1+n-i-1, &{\rm ~if~}\:0\le i\le n-q-1,\\
-s_j=-s_1+j-1, &{\rm ~if~}\:1\le j\le q.
\end{array}	
\right.
\]
Then we have the follows.
\[	
\left\{
\begin{array}{ll}
	&-s_n\geq -s_1+n-1\geq s_1,\\
	&-s_{n-1}\geq-s_1+n-2\geq s_2,\\
	&\quad\quad\vdots\\
 &-s_{q+1}\geq-s_1+q-1\geq s_{n-q},\\
  &-s_{q}=-s_1+q-1=-s_1+n-1-(n-q+1)+1\\
  &\quad \quad \geq -s_1+2z-(n-q+1)+1\geq s_{n-q+1},\\
  &\quad\quad\vdots\\
	&-s_2=-s_1+1=-s_1+n-1-n+2\geq -s_1+2z-(n-1)+1\geq s_{n-1},\\
	&-s_1=-s_1+n-1-n+1\geq -s_1+2z-n+1\geq s_n.	
\end{array}	
\right.
\]	
Recall that
\begin{equation*}
	2q_2^{\odd}=\begin{cases}
		q_2+1,&\text{ if } q_2 \text{ is odd},\\
		q_2,&\text{ if } q_2 \text{ is even},
	\end{cases}
	\quad 2q_2^{\ev}+1=\begin{cases}
		q_2,&\text{ if } q_2 \text{ is odd},\\
		q_2+1,&\text{ if } q_2 \text{ is even}.
	\end{cases}
\end{equation*}
Thus $k(\lambda)\geq q_2(\lambda)\geq  n$ by Lemma \ref{antichain} and Proposition \ref{BCD}. We must have  $k(\lambda)=n$ since $k(\lambda)\leq r=n$.

\item When $z\in\mathbb{Z}$ and $\frac{n}{2}\le z=z_k=n-\frac{k}{2}\le \frac{1}{2}(q+r)$ for $1\le k\le n$, we will have
	\[	
\left\{
\begin{array}{ll}
	s_m\geq -s_1+n-m+1, &{\rm ~if~}\:1\le m\le q,\\	     	  	
	s_l\le -s_1+2z-l+1, &{\rm ~if~}\:q+1\le l\le n,\\
	-s_{n-i}\geq -s_1+n-i-1, &{\rm ~if~}\:0\le i\le n-q-1,\\
	-s_j=-s_1+j-1, &{\rm ~if~}\:1\le j\le q.
\end{array}	
\right.
\]	
\begin{enumerate}[leftmargin=18pt]
\item  When  $q+1\le n-k+2\le n$ (equivalently $2\leq k\leq n-q+1$) and  $n$ is even, we have the follows.
\begin{align}\label{1}
\begin{cases}
&s_{n-(k-1)+1}\le -s_1+2z-(n-k+2)+1=-s_1+n-1\le -s_n,\\
&s_{n-(k-1)+2}\le -s_1+2z-(n-k+3)+1=-s_1+n-2\le -s_{n-1},\\
	&\quad\quad \vdots\\
	&s_{n-1}\le -s_1+2z-(n-1)+1=-s_1+n-(k-3)-1\le -s_{n-k+3}, \\
	&s_n\le -s_1+2z-n+1\le -s_{n-(k-2)}=-s_{n-k+2}.
 \end{cases}	
\end{align}

Thus we have $q_2(\lambda)\geq k-1$ by Lemma \ref{antichain}.

If $q_2(\lambda)\geq k+1$, by Lemma \ref{antichain} we will have $s_{n-k+i}\leq -s_{n-i}$ for $0\leq i\leq k$. From $z=z_k=n-\frac{k}{2}\leq\frac{q+r}{2}$, we can get $2n-k\leq q+r$. Now we choose a positive integer $0\leq i_0\leq k$ such that
$$n-k+i_0\leq q\text{~and~}n-i_0\leq r.$$
\begin{enumerate}[leftmargin=18pt]
    \item When $n-k+i_0\leq q \text{~and~} n-i_0\leq q$,  we have $$s_{n-k+i_0}=z-(n-k+i_0)+1\leq -s_{n-i_0}=-(z-(n-i_0)+1),$$
  which implies that $1\leq -1$.  Therefore we get a contradiction!

\item When $n-k+i_0\leq q \text{~and~} q<n-i_0\leq r$,  we have $$s_{n-k+i_0}=z-(n-k+i_0)+1\leq -s_{n-i_0}=-(z-(n-i_0)),$$
  which implies that $1\leq 0$.  Therefore we get a contradiction!

\end{enumerate}
To sum up,  $k+1$ does not satisfy Lemma \ref{antichain} and    we must have $k-1\leq q_2(\lambda)\leq k$.

If  $q_2(\lambda)= k-1$,  $z=z_k=n-\frac{k}{2}\in \mathbb{Z}$ implies that $k$ is even. Thus $q_2(\lambda)=k-1$ is odd and $k(\lambda)=2q_2^{\odd}=q_2+1=k$ by Proposition \ref{BCD}.

If $q_2(\lambda)=k$, by Proposition \ref{BCD} we have $k(\lambda)=2q_2^{\odd}=q_2=k$ since $q_2(\lambda)=k$ is even. 

 Similarly the result is the same when $n$ is odd, and we omit the process here.

% $$n-k+i_0\leq\frac{q+r}{2}\text{~and~}n-i_0\leq\frac{q+r}{2}.$$
% \begin{enumerate}[leftmargin=15pt]
%     \item When $n-k+i_0\leq q \text{~and~} n-i_0\leq q$, then we have $$s_{n-k+i_0}=z-(n-k+i_0)+1\leq -s_{n-i_0}=-(z-(n-i_0)+1).$$
%   Which implies that $1\leq -1$ and we get a contradiction!

% \item When $n-k+i_0\leq q \text{~and~} q< n-i_0\leq \frac{q+r}{2}$, we have
%  $$z-(n-k+i_0)+1\leq -(z+\lambda_{n-i_0}+n-(n-i_0-1)).$$
% Which implies that $1\leq 0$ since $\lambda_{n-i_0}=-n-1$ and we get a contradiction! 

% \item When $q< n-k+i_0\leq \frac{q+r}{2} \text{~and~}  n-i_0\leq q$, we have
%  $$z+\lambda_{n-k+i_0}+n-(n-k+i_0-1)\leq -(z-(n-i_0)+1).$$
% Which implies that $0\leq -1$ since $\lambda_{n-k+i_0}=-n-1$ and we get a contradiction!

% \item When $q< n-k+i_0\leq \frac{q+r}{2} \text{~and~}  q< n-i_0\leq \frac{q+r}{2}$, we have
% $$z+\lambda_{n-k+i_0}+n-(n-k+i_0-1)\leq-(z+\lambda_{n-i_0}+n-(n-i_0-1)).$$
% Which implies that $0\leq 0$ since $\lambda_{n-k+i_0}=\lambda_{n-i_0}=-n-1$.
% \end{enumerate}
% To sum up,  $k+1$ does not satisfy Lemma \ref{antichain} and    we have $k-1\leq q_2(\lambda)\leq k$.

% If  $q_2(\lambda)= k-1$, now $z=z_k=n-\frac{k}{2}\in \mathbb{Z}$ implies that $k$ is even. Thus $q_2=k-1$ is odd and $k(\lambda)=2q_2^{\odd}=q_2+1=k$ by Proposition \ref{BCD}.

% If $q_2(\lambda)= k$, we have $k(\lambda)=2q_2^{\odd}=2\cdot \frac{1}{2}k=k$ since $k$ is even. 

%  Similarly the result is the same when $n$ is odd, and we omit the process here.

\item When $2\leq n-k+2\leq q$ (equivalently $n-q+2\leq k\leq n$) and $n$ is even, we have the follows.
\begin{align}\label{12}
\begin{cases}
&-s_{n-(k-1)+1}= -s_1+n-k+1=-s_1+2z-n+1\geq s_n,\\
&-s_{n-(k-1)+2}= -s_1+n-k+2=-s_1+2z-(n-1)+1\geq s_{n-1},\\
	&\quad\quad \vdots\\
	&-s_{n-1}\geq -s_1+n-2=-s_1+2z-(n-k+3)+1\geq s_{n-k+3}, \\
	&-s_n\geq -s_1+n-1=-s_1+2z-(n-k+2)+1\geq s_{n-k+2}.
 \end{cases}	
\end{align}
Then we have  $q_2(\lambda)\geq k-1$ by Lemma \ref{antichain}.   Similarly in this case, we can obtain $k(\lambda)=k$ by Proposition \ref{BCD}. The result is the same when $n$ is odd, and we omit the process here.

\item When $z=z_1=n-\frac{1}{2}\leq \frac{1}{2}(q+r)\leq n$, we will have $q=n-1<r=n$ or $q=r=n$. Thus
$\lambda$ is half-integral and $\lambda_1=\cdots=\lambda_{n-1}=-n>\lambda_n=-n-1$ or $\lambda_1=\cdots=\lambda_{n}=-n$.  So  $s_{n-1}=\frac{3}{2}>s_n=n-\frac{1}{2}+\lambda_n+1=-\frac{1}{2}$ or $\frac{1}{2}$. By using R-S algorithm, we have $q_2(\lambda)=1$ if $s_n=-\frac{1}{2}$ and  $q_2(\lambda)=0$ if $s_n=\frac{1}{2}$.  
Therefore $k(\lambda)=1=k$ by Proposition \ref{BCD}.

\item When $z=z_0=n\leq \frac{1}{2}(q+r)\leq n$, we will have $q=r=n$. Thus
$\lambda$ will be dominant integral since $\lambda_1=\cdots=\lambda_n=-n$ and  $s_n=n+\lambda_n+1=1$. Therefore $k(\lambda)=0=k$ by Proposition \ref{BCD}.

\end{enumerate}

\item  When  $z\in\frac{1}{2}+\mathbb{Z}$ and $\frac{n}{2}\le z=z_k=n-\frac{k}{2}\le \frac{1}{2}(q+r)$ for $1\le k\le n$, we still have
$q_2(\lambda)=k-1$ by (\ref{1}) and (\ref{12}). In this case $q_2(\lambda)=k-1$ and $k(\lambda)=2q_2^{\ev}+1=k$ by Proposition \ref{BCD}.
\end{enumerate}

% Similarly we can prove the case where $n$ is odd, and the result is the same, we omit the proof here.

To sum up, we have $k(\lambda)=k$, where $k(\lambda)$ is given in Proposition \ref{BCD}. Therefore we have completed the proof of the case when $G=Sp(n,\mathbb{R})$. 

\subsection{Proof for $G=SO^{\ast}(2n)$} 
 For $SO^{\ast}(2n)$, from \cite{EHW} the root system $Q(\lambda_0)=R(\lambda_0)$ is of type $\mathfrak{su}(1,q)$ with $1\le q \le n-1$ or  $\mathfrak{so}^{\ast}(2p)$ with $3\le p\le n$. From Table \ref{constants}, we have $c=2$ and $(\rho,\beta^{\vee})=2n-3$.
 $(\lambda_0+\rho,\beta^{\vee})=0$ implies that $\lambda_1+\lambda_2=-2n+3$.

\begin{enumerate}[leftmargin=18pt]
	\item When $Q(\lambda_0)=R(\lambda_0)$ is of type $\mathfrak{su}(1,q)$ with $1\le q \le n-1$, we have 
	$$\lambda_1>\lambda_2=\lambda_3=\cdots=\lambda_{q+1}>\lambda_{q+2}\geq \cdots \geq \lambda_n.$$
From \cite{EHW} we know that $L(\lambda)$ is unitarizable if and only if $z\le q$. Thus we can write
\begin{align*}
	\lambda&=\lambda_0+z\xi+\rho\\
 &=(\frac{1}{2}z+\lambda_1+n-1,\frac{1}{2}z+\lambda_2+n-2,\cdots,\frac{1}{2}z+\lambda_{q+1}+n-q-1,\\
 &\quad \: \frac{1}{2}z+\lambda_{q+2}+n-q-2,\cdots,\frac{1}{2}z+\lambda_n)\\
 &=(\frac{1}{2}z-\lambda_2-n+2,\frac{1}{2}z+\lambda_2+n-2,\cdots,\frac{1}{2}z+\lambda_{2}+n-q-1,\\
  &\quad \: \frac{1}{2}z+\lambda_{q+2}+n-q-2,\cdots,\frac{1}{2}z+\lambda_n)\\
  &:=(t_1,\cdots,t_{q+1},t_{q+2},\cdots,t_n).
\end{align*}	
So $(t_1,t_2,\cdots,t_n)$ and $(-t_n,-t_{n-1},\cdots,-t_1)$ are strictly decreasing sequences, and it is not difficult to get
\[	
\left\{
\begin{array}{ll}
	t_1> -t_1+z+1, \\	     	  	
	t_i\le -t_1+z+2-i, &{\rm ~if~}\:2\le i\le n,\\	
	-t_j\geq -t_1+j-1, &{\rm ~if~}\:1\le j\le n.
\end{array}	
\right.
\]		
	
When $L(\lambda)$ is a unitary highest weight module and  $n$ is even, we have the follows.
\begin{enumerate}[leftmargin=18pt]
	\item When $z< q\leq n-1=z_{r-1}$, we have $L(\lambda)=N(\lambda)$ by \cite{EHW}. Thus $k(\lambda)=\frac{n}{2}=k$. 
 
 \item When $z= q< n-1=z_{r-1}$, we have 
 \begin{align*}
\begin{cases}
	& t_{1}=\frac{1}{2}z-\lambda_2-n+2\le \frac{1}{2}z-\lambda_n-n+2\leq -\frac{1}{2}z-\lambda_n=-t_n,\\
& t_2=\frac{1}{2}z+\lambda_2+n-2=\frac{1}{2}z-\lambda_1-2n+3+n-2\leq -\frac{1}{2}z-\lambda_{n-1}-1=-t_{n-1},\\
	 &\quad\quad\vdots\\
	& t_n\leq -t_1.
 \end{cases}
\end{align*}	
 Thus $k(\lambda)= q^{\ev}_2=[\frac{q_2}{2}]\geq \frac{n}{2}$ by Lemma \ref{antichain} and Proposition \ref{BCD}. We must have  $k(\lambda)=\frac{n}{2}=r$ since $k(\lambda)\leq r=[\frac{n}{2}]=\frac{n}{2}$ and $n$ is even.

%  we will have
% \begin{align}\label{2}
% ???\begin{cases}
% 	& t_{n-\frac{n}{2}+1}\le -t_1+z+2-\frac{n}{2}-1=-t_1+z-\frac{n}{2}+1\le -t_1+n-1\leq -t_n,\\
% 	&t_{n-\frac{n}{2}+2}\le -t_1+z+2-\frac{n}{2}-2\le -t_1+n-1\leq -t_{n-1},\\	
%     &\quad\quad\vdots\\
% 	&t_{n-1}\le -t_1+z+3-n\le -t_1+\frac{n}{2}+1=-t_1+\frac{n}{2}+2-1\le -t_{n-\frac{n}{2}+2},\\
% 	& t_{n}\le -t_1+z+2-n\le -t_1+\frac{n}{2}+2\le -t_{n-\frac{n}{2}+1}.
%  \end{cases}
% \end{align}	
% Thus $k(\lambda)= q^{\ev}_2=[\frac{q_2}{2}]\geq \frac{n}{2}$ by Lemma \ref{antichain} and Proposition \ref{BCD}. We must have  $k(\lambda)=\frac{n}{2}=k$ since $k(\lambda)\leq r=[\frac{n}{2}]=\frac{n}{2}$ since $n$ is even.

\item When $z=z_k=2n-3-2k=q\leq n-1$ for $1\le k \le r-1=\frac{n}{2} -1$, we will have  $z=z_{r-1}=n-1=q$ and $\lambda_1>\lambda_2=\lambda_3=\cdots= \lambda_n$.  Thus we can get 
 \begin{align*}
\begin{cases}
& t_2=\frac{1}{2}z+\lambda_2+n-2=\frac{1}{2}z-\lambda_1-n+1\leq -\frac{1}{2}z-\lambda_{n}=-t_{n},\\
	 &\quad\quad\vdots\\
	& t_n\leq -t_2.
 \end{cases}
\end{align*}	
Thus $q_2(\lambda)\geq n-1$ by Lemma \ref{antichain}.
Note that $$t_1=\frac{1}{2}z-\lambda_2-n+2=-\frac{1}{2}z-\lambda_{n}+1=-t_n+1>-t_n.$$
So we can not have $q_2(\lambda)=n$. Thus $q_2(\lambda)=n-1$ and  $k(\lambda)= q^{\ev}_2=[\frac{q_2}{2}]\geq \frac{n}{2}$ by Lemma \ref{antichain} and Proposition \ref{BCD}.

% \begin{align}\label{3}
% \begin{cases}
% 	& t_{n-(2k+1)+1}\le -t_1+z+2-(n-2k)=-t_1+z-n+2k+2\le -t_n,\\
% 		& t_{n-(2k+1)+2}\le -t_1+z+2-(n-2k+1)=-t_1+n-2\le -t_{n-1},\\
% 	 &\quad\quad\vdots\\
% 	& t_{n-1}\le -t_1+z+2-(n-1)=-t_1+n-2k\le -t_{n-(2k+1)+2},\\
% 	&t_{n}\le -t_1+z+2-n=-t_1+n-2k-1\le -t_{n-(2k+1)+1}.
%  \end{cases}
% \end{align}	
% ???In this case $2k+2$ does not satisfy Lemma \ref{antichain}, therefore $q_2=2k+1$ by Lemma \ref{antichain}. ???
% One can easily check that
% \begin{equation}\label{4}
% 	q_2^{\ev}=\left\lfloor \frac{q_2}{2} \right\rfloor=\begin{cases}
% 		\frac{q_2-1}{2}&\text{ if } q_2 \text{ is odd},\\
% 		\frac{q_2}{2}&\text{ if } q_2 \text{ is even}.
% 	\end{cases}
% \end{equation}

Thus $k(\lambda)= q^{\ev}_2=[\frac{q_2}{2}]= k$ by Lemma \ref{antichain} and Proposition \ref{BCD}.

\end{enumerate}	
 The argument is similar when $n$ is odd, so we omit the details here.  Therefore, when $Q(\lambda_0)=R(\lambda_0)$ is of type $\mathfrak{su}(1,q)$ with $1\le q \le n-1$, we have $k(\lambda)=k$.

	\item When $Q(\lambda_0)=R(\lambda_0)$ is of type $\mathfrak{so}^{\ast}(2p)$ with $3\le p\le n$, we have
$$\lambda_1=\lambda_2=\cdots=\lambda_p=-n+\frac{3}{2}>\lambda_{p+1}\geq \cdots \geq \lambda_n.$$	
From \cite{EHW} we know that $L(\lambda)$ is unitarizable if and only if \begin{equation*}
	z\le\begin{cases}
		p,&\text{ if } p \text{ is odd},\\
		p-1,&\text{ if } p \text{ is even},
	\end{cases}
\end{equation*} or $z=2p-3-2j\leq 2p-3$ for some integer $0\le j \le [\frac{p}{2}]-2$. 

Thus we can write
\begin{align*}
	\lambda&=\lambda_0+z\xi+\rho\\
 &=(\frac{1}{2}z+\lambda_1+n-1,\cdots,\frac{1}{2}z+\lambda_p+n-p,\frac{1}{2}z+\lambda_{p+1}+n-p-1,\cdots,\frac{1}{2}z+\lambda_n)\\
 &=(\frac{1}{2}z+\frac{1}{2},\cdots,\frac{1}{2}z+\frac{3}{2}-p,\frac{1}{2}z+\lambda_{p+1}+n-p-1,\cdots,\frac{1}{2}z+\lambda_n)\\
 &:=(t_1,\cdots,t_p,t_{p+1},\cdots,t_n).
\end{align*}
So $(t_1,t_2,\cdots,t_n)$ and $(-t_n,-t_{n-1},\cdots,-t_1)$ are strictly decreasing sequences, and it is not difficult to get
\[	
\left\{
\begin{array}{ll}     	  	
	t_i\le -t_1+z+2-i, &{\rm ~if~}\:1\le i\le n,\\	
	-t_j\geq -t_1+j-1, &{\rm ~if~}\:1\le j\le n.
\end{array}	
\right.
\]	
Suppose $p$ and $n$ are even. Then we have the follows.	
\begin{enumerate}[leftmargin=18pt]
	\item When $z< p-1\leq n-1=z_{r-1}$, by similar arguments as in the case of $\mathfrak{su}(1,q)$ we have $q_2(\lambda)= n$. Thus $k(\lambda)= q^{\ev}_2=[\frac{q_2}{2}]=[\frac{n}{2}]=r$ by Lemma \ref{antichain} and Proposition \ref{BCD}. 	

	\item When $z\geq p-1$ and $z=z_k=2n-3-2k\leq 2p-3$ for some $1\le k\le r-1=\left\lfloor \frac{n}{2} \right\rfloor-1$, we have 
 	
\begin{align*}
\begin{cases}
	& t_{n-(2k+1)+1}\le -t_1+z+2-(n-2k)=-t_1+z-n+2k+2\le -t_n,\\
		& t_{n-(2k+1)+2}\le -t_1+z+2-(n-2k+1)=-t_1+n-2\le -t_{n-1},\\
	 &\quad\quad\vdots\\
	& t_{n-1}\le -t_1+z+2-(n-1)=-t_1+n-2k\le -t_{n-(2k+1)+2},\\
	&t_{n}\le -t_1+z+2-n=-t_1+n-2k-1\le -t_{n-(2k+1)+1}.
 \end{cases}
\end{align*}	
 Thus we have $q_2(\lambda)\geq 2k+1$ by Lemma \ref{antichain}. 
 
 If $q_2(\lambda)\geq 2k+2$, by Lemma \ref{antichain} we will have $t_{n-2k-2+i}\leq -t_{n-i+1}$ for $1\leq i\leq 2k+2$.
  From $ z=z_k=2n-3-2k\leq  2p-3$, we can get $$ 2n-2k-1< 2p\Rightarrow n-2k-2+n+1<2p.$$
% Thus we can get
% $$p-k\leq p' \text{~and~} p+1>n-q'$$
Now we choose a  positive integer $1\leq i_0\leq 2k+2$ such that $$n+1-i_0\leq p ~\text{~and~} n-2k-2+i_0\leq p. $$
Then we have \begin{align*}
    &t_{n-2k-2+i_0}=\frac{1}{2}z+\frac{1}{2}-(n-2k-2+i_0-1)\\
    &\leq -t_{n+1-i_0}
    =-(\frac{1}{2}z+\frac{1}{2}-(n+1-i_0-1)),
\end{align*}
which implies that
$$z-2n+2k+4\leq 0\Rightarrow 1\leq 0.$$
Therefore we get contradiction! So we must have $q_2(\lambda)= 2k+1$.

 Thus $k(\lambda)= q^{\ev}_2=[\frac{q_2}{2}]= k$ by Lemma \ref{antichain} and Proposition \ref{BCD}. 	
\end{enumerate}
 The argument is similar when $n$ is odd, so we omit the details here. 
\end{enumerate}

 % We have proved that Proposition \ref{7.2} holds for $G=Sp(n,\mathbb{R})$ and $SO^{\ast}(2n)$, and the arguments for $G=SU(p,n-p)$ are similar, we omit the proof here. 

\subsection{Proof for $G=SO(2,2n-2)$} For $SO(2,2n-2)$, from \cite{EHW} the root system $Q(\lambda_0)=R(\lambda_0)$ is of type $\mathfrak{su}(1,p)$ with $1\le p\le n-1$ or $\mathfrak{so}(2,2n-2)$. From Table \ref{constants}, we have $c=n-2$ and $(\rho,\beta^{\vee})=2n-3$. $(\lambda_0+\rho,\beta^{\vee})=0$ implies that $\lambda_1+\lambda_2=-2n+3$.

\begin{enumerate}[leftmargin=18pt]
    \item When $Q(\lambda_0)=R(\lambda_0)$ is of type $\mathfrak{su}(1,p)$ with $1\le p\le n-1$, we have
    $$|\lambda_n|\leq \lambda_{n-1}\leq \cdots \leq \lambda_{p+2}<\lambda_{p+1}=\lambda_p=\cdots=\lambda_{2}\in \frac{1}{2}\mathbb{Z}_{>0}.$$

From \cite{EHW} we know that $L(\lambda)$ is unitarizable if and only if $z\le p$. From Theorem \ref{so(2,2n-2)}, we have
$k(\lambda)=2$ if and only if $\lambda$ is not integral or $$z\leq -|\lambda_n|-\lambda_1-n+1=-|\lambda_n|+\lambda_2+2n-3-n+1=-|\lambda_n|+\lambda_2+n-2.$$
Now  when $z< n-1=z_1$,  $z$ will be not integral,  or $z\in \mathbb{Z}$  and $z\leq n-2\leq -|\lambda_n|+\lambda_2+n-2$.
Thus we have $k(\lambda)=2$. Note that when $p\leq n-2$, we always have $z\leq p\leq n-1$.

Now when $z= n-1=z_1$, we must have $p=n-1$ since $z\leq p\leq n-1$. Then we have $\lambda_2=\cdots=\lambda_{n-1}=|\lambda_n|\in \frac{1}{2}\mathbb{Z}_{>0}$. From Theorem \ref{so(2,2n-2)}, we have
$k(\lambda)=1$ if and only if $\lambda$ is  integral and  $$ -|\lambda_n|-\lambda_1-n+1=-|\lambda_n|+\lambda_2+n-2<z\leq \lambda_2 -\lambda_1-1,$$
equivalently we have 
$$n-2<z\leq 2\lambda_2+2n-4.$$
Thus  when $z= n-1=z_1$,  
 we will  have $k(\lambda)=1$.

\item When $Q(\lambda_0)=R(\lambda_0)$ is of type $\mathfrak{so}(2,2n-2)$, we have
	$$\lambda_2=\lambda_3=\cdots=\lambda_n=0.$$
From \cite{EHW} we know that $L(\lambda)$ is unitarizable if and only if $z\le n-1=z_1$ or $z=2n-3=z_0$. In this case, we have  $\lambda_1=-2n+3$.

From Theorem \ref{so(2,2n-2)}, we have
$k(\lambda)=2$ if and only if $\lambda$ is not integral or $$z\leq -|\lambda_n|-\lambda_1-n+1=n-2.$$
Now  when $z< n-1=z_1$,  $z$ will be not integral,  or $z\in \mathbb{Z}$  and $z\leq n-2$.
Thus we have $k(\lambda)=2$.

From Theorem \ref{so(2,2n-2)}, we have
$k(\lambda)=1$ if and only if $\lambda$ is  integral and  $$ -|\lambda_n|-\lambda_1-n+1<z\leq \lambda_2 -\lambda_1-1,$$
equivalently we have 
$$n-2<z\leq 2n-4.$$
Thus  when $z= n-1=z_1$,  
 we will  have $k(\lambda)=1$.

From Theorem \ref{so(2,2n-2)}, we have
$k(\lambda)=0$ if and only if $\lambda$ is  integral and  $$z\geq \lambda_2 -\lambda_1=2n-3.$$
Thus  when $z= 2n-3=z_0$,  
 we will  have $k(\lambda)=0$.

\end{enumerate}
Now we complete the proof for all cases.

\subsection{Proof for $G=SO(2,2n-1)$}  For $SO(2,2n-1)$, from \cite{EHW} the root system $Q(\lambda_0)=R(\lambda_0)$ is of type $\mathfrak{su}(1,p)$ with $1\le p\le n-1$ or $\mathfrak{so}(2,2n-1)$, or $Q(\lambda_0)=\mathfrak{su}(1,n-1)$ and $R(\lambda_0)=\mathfrak{so}(2,2n-1)$. From Table \ref{constants}, we have $c=n-\frac{3}{2}$ and $(\rho,\beta^{\vee})=2n-2$. $(\lambda_0+\rho,\beta^{\vee})=0$ implies that $\lambda_1+\lambda_2=-2n+2$.
\begin{enumerate}[leftmargin=18pt]
  \item When $Q(\lambda_0)=R(\lambda_0)$ is of type $\mathfrak{su}(1,p)$ with $1\le p\le n-1$, we have
$$0\leq \lambda_n\leq \lambda_{n-1}\leq\cdots\leq \lambda_{p+2}<\lambda_{p+1}=\lambda_p=\cdots=\lambda_{2}.$$
From \cite{EHW} we know that $L(\lambda)$ is unitarizable if and only if $z\le p$. 

From Theorem \ref{Bpartition}, we have $k(\lambda)=2$ if and only if $z\notin \frac{1}{2}\mathbb{Z}$, or $z\in \mathbb{Z}$ and $z< \lambda_2-\lambda_1$, or $z\in\frac{1}{2}+\mathbb{Z}$ and $z\leq -\lambda_1-n+\frac{1}{2}$.

Now $z\leq p< n-\frac{1}{2}=z_1$, we will have  $z\notin \frac{1}{2}\mathbb{Z}$, or $z\in\mathbb{Z}$ and $z<n-\frac{1}{2}<2n-2\leq\lambda_2-\lambda_1$, or $z\in\frac{1}{2}+\mathbb{Z}$ and $z\leq n-\frac{3}{2}<\lambda_2+n-\frac{3}{2}=-\lambda_1-n+\frac{1}{2}$. Thus we have $k(\lambda)=2$ by Theorem \ref{Bpartition}.

% From Theorem \ref{Bpartition}, we have $k(\lambda)=1$ if and only if $z\in \frac{1}{2}+\mathbb{Z}$ and 
% $$z> -\lambda_1-n+\frac{1}{2}.$$
% Now we have $\lambda_1>0>-2n+1$, equivalently $$z=n-\frac{1}{2}> -\lambda_1-n+\frac{1}{2}.$$
% Thus when $z= n-\frac{1}{2}=z_1$, we will have $k(\lambda)=1$.

\item  When $Q(\lambda_0)=R(\lambda_0)$ is of type $\mathfrak{so}(2,2n-1)$, we have
$$\lambda_2=\lambda_3=\cdots=\lambda_n=0.$$
From \cite{EHW} we know that $L(\lambda)$ is unitarizable if and only if $z\le n-\frac{1}{2}=z_1$ or $z=2n-2=z_0$. In this case, we have  $\lambda_1=-2n+2$.

From Theorem \ref{Bpartition}, we have $k(\lambda)=2$ if and only if $z\notin \frac{1}{2}\mathbb{Z}$, or $z\in \mathbb{Z}$ and $z< \lambda_2-\lambda_1=2n-2$, or $z\in\frac{1}{2}+\mathbb{Z}$ and $z\leq -\lambda_1-n+\frac{1}{2}=n-\frac{3}{2}$.

Now when $z< n-\frac{1}{2}=z_1$, we will have 
$z\notin \frac{1}{2}\mathbb{Z}$,
 or $z\in\mathbb{Z}$ and $z<n-\frac{1}{2}<2n-2$, or $z\in\frac{1}{2}+\mathbb{Z}$ and $z\leq n-\frac{3}{2}$. Thus we have $k(\lambda)=2$.

From Theorem \ref{Bpartition}, we have $k(\lambda)=1$ if and only if $z\in \frac{1}{2}+\mathbb{Z}$ and  $z>-\lambda_1-n+\frac{1}{2}=n-\frac{3}{2}$.
Thus when
$z=z_1=n-\frac{1}{2}>n-\frac{3}{2}$, we will have $k(\lambda)=1$.

From Theorem \ref{Bpartition}, we have $k(\lambda)=0$ if and only if $z\in \mathbb{Z}$ and $z\geq \lambda_2-\lambda_1=2n-2$. Thus when
$z=z_0=2n-2$, we will have $k(\lambda)=0$.

% When $z= n-\frac{1}{2}=z_1$, by Theorem \ref{Bpartition} we have $k(\lambda)=1$ if and only if $\lambda$ is half integral and 
% $$z=n-\frac{1}{2}\geq -\lambda_1-n+\frac{1}{2}=n-\frac{3}{2}.$$ Thus when $z= n-\frac{1}{2}=z_1$, we will have $k(\lambda)=1$.

% When $z=2n-2=z_0$, by Theorem \ref{Bpartition} we have $k(\lambda)=0$ if and only if $\lambda$ is integral and 
% $$z\geq\lambda_2-\lambda_1=2n-2.$$
% Thus when $z=2n-2=z_0$ we have $k(\lambda)=0$.

\item When $Q(\lambda_0)=\mathfrak{su}(1,n-1)$ and $R(\lambda_0)=\mathfrak{so}(2,2n-1)$, we have 
$$\lambda_2=\lambda_3=\cdots=\lambda_n=\frac{1}{2}.$$
From \cite{EHW} we know that $L(\lambda)$ is unitarizable if and only if $z\le n-\frac{1}{2}$.

From Theorem \ref{Bpartition}, we have $k(\lambda)=2$ if and only if $z\notin \frac{1}{2}\mathbb{Z}$, or $z\in \mathbb{Z}$ and $z< \lambda_2-\lambda_1=2n-1$, or $z\in\frac{1}{2}+\mathbb{Z}$ and $z\leq -\lambda_1-n+\frac{1}{2}=n-1$.

Now when $z< n-\frac{1}{2}=z_1$, we will have 
$z\notin \frac{1}{2}\mathbb{Z}$,
 or $z\in\mathbb{Z}$ and $z<n-\frac{1}{2}<2n-1$, or $z\in\frac{1}{2}+\mathbb{Z}$ and $z\leq n-\frac{3}{2}<n-1$. Thus we have $k(\lambda)=2$.

From Theorem \ref{Bpartition}, we have $k(\lambda)=1$ if and only if $z\in \frac{1}{2}+\mathbb{Z}$ and  $z>-\lambda_1-n+\frac{1}{2}=n-1$.
Thus when
$z=z_1=n-\frac{1}{2}>n-1$, we will have $k(\lambda)=1$.

%  From Theorem \ref{Bpartition}, we have $k(\lambda)=0$ if and only if $z\in \mathbb{Z}$ and $z\geq \lambda_2-\lambda_1=2n-1$. Thus when
% $z=z_0=2n-2$, we will have $k(\lambda)=0$.
\end{enumerate}

Now we complete the proof of all cases.

\subsection{Proof for $G=E_{6(-14)}$} 
 For $E_{6(-14)}$, from \cite[\S 12]{EHW} the root system $Q(\lambda_0)=R(\lambda_0)$ is of type $\mathfrak{su}(1,p)$ with $1\le p \le 5$,  $\mathfrak{so}(2,8)$ or $\mathfrak{e}_{6(-14)}$. For all cases, we have $\lambda_6=\lambda_7=-\lambda_8:=b$ by \cite{Bour}. From Table \ref{constants}, we have $c=3$ and $(\rho,\beta^{\vee})=11$. $(\lambda_0+\rho,\beta^{\vee})=0$ implies that $\lambda_1+\cdots+\lambda_5-\lambda_6-\lambda_7+\lambda_8=-22$.

\begin{enumerate}[leftmargin=18pt]
	\item When $Q(\lambda_0)=R(\lambda_0)$ is of type $\mathfrak{su}(1,5)$, we have
	$$-\lambda_1=\lambda_2=\lambda_3=\lambda_4=\lambda_{5}:=a\in \frac{1}{2}\mathbb{Z}_{\geq 0}.$$
From \cite{EHW} we know that $L(\lambda)$ is unitarizable if and only if $z\le 5$. In this case, we have   $3a-3b=-22$ since $\lambda_1+\cdots+\lambda_5-\lambda_6-\lambda_7+\lambda_8=-22$.

From Theorem \ref{e6-GK}, we have $$-(\lambda_0,\beta_{1,2})=\frac{1}{2}(a+3b)=\frac{1}{2}(a+3a+22)=2a+11$$ and $k(\lambda)=2$ if and only if $\lambda$ is not integral or $z\in\mathbb{Z}$ and $$z\leq -4+\min\{-(\lambda_0,\beta_{1,2})\}=7+2a.$$
Now we have $z\leq 5<7+2a$, so $k(\lambda)=2$.

The proof for other cases of type $\mathfrak{su}(1,p)$ with $1\le p \le 4$ is similar, and we omit the details here.

\item When $Q(\lambda_0)=R(\lambda_0)$ is of type $\mathfrak{so}(2,8)$, we have
	$$\lambda_1=\lambda_2=\lambda_3=\lambda_4=0<\lambda_5\in \mathbb{Z}_{> 0}.$$
From \cite{EHW} we know that $L(\lambda)$ is unitarizable if and only if $z\le 4$ or $z=7$. In this case, we have  $\lambda_5-3b=-22$.

From Theorem \ref{e6-GK}, we have $$-(\lambda_0,\beta_{1,2})=\frac{1}{2}(a+3b)=\frac{1}{2}(\lambda_5+3b)=\frac{1}{2}(\lambda_5+\lambda_5+22)=\lambda_5+11$$ and $k(\lambda)=2$ if and only if $\lambda$ is not integral or $z\in\mathbb{Z}$ and $$z\leq -4+\min\{-(\lambda_0,\beta_{1,2})\}=7+\lambda_5.$$
Now we have $z\leq 7<7+\lambda_5$, so $k(\lambda)=2$.

\item When $Q(\lambda_0)=R(\lambda_0)$ is of type $\mathfrak{e}_{6(-14)}$, we have
	$$\lambda_1=\lambda_2=\lambda_3=\lambda_4=\lambda_5=0.$$
From \cite{EHW} we know that $L(\lambda)$ is unitarizable if and only if $z\le 8=z_1$ or $z=11=z_0$. In this case, we have  $3b=22$.

From Theorem \ref{e6-GK}, we have $$-(\lambda_0,\beta_{1,2})=\frac{1}{2}(0+3b)=11$$ and $k(\lambda)=2$ if and only if $\lambda$ is not integral or $z\in\mathbb{Z}$ and $$z\leq -4+\min\{-(\lambda_0,\beta_{1,2})\}=7.$$

Now when $z< 8=z_1$,  $z$ will be not integral,  or $z\in \mathbb{Z}$  and $z\leq 7$.
Thus we have $k(\lambda)=2$.

From Theorem \ref{e6-GK}, we have $$-(\lambda_0,\alpha_{1})=\frac{1}{2}(0+3b)=11$$ and $k(\lambda)=1$ if and only if $z\in\mathbb{Z}$ and $$7=-4+\min\{-(\lambda_0,\beta_{1,2})\}<z\leq -(\lambda_0,\alpha_{1})-1=10.$$
Thus when $z=z_1=8$, we will have $k(\lambda)=1$.

From Theorem \ref{e6-GK}, we have $k(\lambda)=0$ if and only if $z\in\mathbb{Z}$ and $$z> -(\lambda_0,\alpha_{1})-1=10.$$
Thus when $z=z_0=11$, we will have $k(\lambda)=0$.

\end{enumerate}
Now we complete the proof for all cases.

\subsection{Proof for $G=E_{7(-25)}$} 
 For $E_{7(-25)}$, from \cite[\S 13]{EHW} the root system $Q(\lambda_0)=R(\lambda_0)$ is of type $\mathfrak{su}(1,p)$ with $1\le p \le 6$,  $\mathfrak{so}(2,10)$ or $\mathfrak{e}_{7(-25)}$. For all cases, we have $\lambda_7=-\lambda_8$ by \cite{Bour}. From Table \ref{constants}, we have $c=4$ and $(\rho,\beta^{\vee})=17$. $(\lambda_0+\rho,\beta^{\vee})=0$ implies that  $\lambda_8=-\frac{17}{2}$.
\begin{enumerate}[leftmargin=18pt]
    \item When $Q(\lambda_0)=R(\lambda_0)$ is of type $\mathfrak{su}(1,6)$, we have
$$0<\lambda_1=\lambda_2=\lambda_3=\lambda_4=\lambda_5:=a\in\frac{1}{2}\mathbb{Z}\text{~and~} \lambda_8-(\sum\limits_{i=2}^7\lambda_i)+\lambda_1=0.$$
From \cite{EHW} we know that $L(\lambda)$ is unitarizable if and only if $z\le 6$. In this case we have $\lambda_6=-17-3a$.

From Theorem \ref{e7-GK} we have $k(\lambda)=3$ if and only if $\lambda$ is not integral or $z\in\mathbb{Z}$ and 
$$z\leq -9+\min\{-(\lambda_0,\beta_{3,4,5})\}=8+2a.$$
Now we have $z\leq 6<8+2a$,  so $k(\lambda)=3$.  

The proof for other cases of type $\mathfrak{su}(1,p)$ with $1\le p \le 5$ is similar, and we omit the details here.

     \item When $Q(\lambda_0)=R(\lambda_0)$ is of type $\mathfrak{so}(2,10)$, we have 
$$\lambda_1=\lambda_2=\lambda_3=\lambda_4=0<\lambda_5\in\mathbb{N}^{\ast}\text{~and~} \lambda_8-(\sum\limits_{i=2}^7\lambda_i)+\lambda_1=0.$$
From \cite{EHW} we know that $L(\lambda)$ is unitarizable if and only if $z\le 5=z_3$ or $z=9=z_2$. In this case $\lambda_5+\lambda_6=\lambda_8-\lambda_7=-17$.

From Theorem \ref{e7-GK} we have 
$$-(\lambda_0,\beta_{3,4})=-\frac{1}{2}(-\lambda_5+\lambda_6-\lambda_7+\lambda_8)=\lambda_5+17,-(\lambda_0,\beta_{5})=-(\lambda_5+\lambda_6)=17$$ and $k(\lambda)=3$ if and only if $\lambda$ is not integral or $z\in\mathbb{Z}$ and 
$$z\leq -9+\min\{-(\lambda_0,\beta_{3,4,5})\}=8.$$
Now when $z<9=z_2$,  $z$ will be not integral,  or $z\in \mathbb{Z}$  and $z\leq 8$.
Thus we have $k(\lambda)=3$.

% Then $z\leq 5\leq 8$ and hence $k(\lambda)=3$.

From Theorem \ref{e7-GK} we have  $k(\lambda)=2$ if and only if  $z\in\mathbb{Z}$ and $$8=-9+\min\{-(\lambda_0,\beta_{3,4,5})\}<z\le -(\lambda_0,\beta_{6,7})-5=\lambda_5+12.$$
Thus when $z=9=z_2$, we will have $k(\lambda)=2$.

% We have $8\leq z=9\leq \lambda_5+12$ and $k(\lambda)=2$.

\item When $Q(\lambda_0)=R(\lambda_0)$ is of type $\mathfrak{e}_{7(-25)}$, we have   $$\lambda_1=\lambda_2=\lambda_3=\lambda_4=\lambda_5=0\text{~and~} \lambda_8-(\sum\limits_{i=2}^7\lambda_i)+\lambda_1=0.$$
From \cite{EHW} we know that $L(\lambda)$ is unitarizable if and only if $z\le 9=z_2$ or $z=13=z_1$ or $z=17=z_0$. In this case we have $\lambda_6=-17$.

From Theorem \ref{e7-GK} we have $k(\lambda)=3$ if and only if $z$ is not integral or $z\in\mathbb{Z}$ and 
$$z\leq -9+\min\{-(\lambda_0,\beta_{3,4,5})\}=8.$$
Now when $z<9=z_2$, $z$ will be not integral, or $z\in\mathbb{Z}$ and $z\leq 8$, hence we have $k(\lambda)=3$.

From Theorem \ref{e7-GK} we have  $k(\lambda)=2$ if and only if  $z\in\mathbb{Z}$ and $$8=-9+\min\{-(\lambda_0,\beta_{3,4,5})\} <z\le -(\lambda_0,\beta_{6,7})-5=12.$$
     Thus when $z=9=z_2$, we will have $k(\lambda)=2$.

 From Theorem \ref{e7-GK} we have  $k(\lambda)=1$ if and only if  $z\in\mathbb{Z}$ and $$12=-(\lambda_0,\beta_{6,7})-5< z\le -(\lambda_0,\beta_8)-1=16.$$    
Thus when $z=13=z_1$, we will have $k(\lambda)=1$.

From Theorem \ref{e7-GK} we have  $k(\lambda)=0$ if and only if  $z\in\mathbb{Z}$ and $$z\geq -(\lambda_0,\beta_8)-1=16.$$
Thus when $z=17=z_0$, we will  have $k(\lambda)=0$.
 
\end{enumerate}

So far, we have completed the proof of all cases of Theorem \ref{kconst}.

	\subsection*{Acknowledgments}
	
	Z. Bai was supported in part by the National Natural Science Foundation of 
	China (No. 12171344) and the National Key $\textrm{R}\,\&\,\textrm{D}$ Program of China (No. 2018YFA0701700 and No. 2018YFA0701701).

\printbibliography
\end{document}